\documentclass[11pt]{article}

\usepackage{amsmath,amssymb,amsthm}
\usepackage[unicode,breaklinks=true,colorlinks=true]{hyperref}
\usepackage[dvipsnames]{xcolor}
\usepackage[normalem]{ulem} %
\usepackage{soul}
\usepackage{marginnote}

\usepackage[top=1in, bottom=1in, left=1.25in, right=1.25in, marginparwidth=1in, marginparsep=0.1in]{geometry}

\numberwithin{equation}{section}
\newtheorem{theorem}{Theorem}[section]

\newtheorem{lemma}[theorem]{Lemma}
\newtheorem{definition}[theorem]{Definition}

\theoremstyle{remark}
\newtheorem{remark}[theorem]{Remark}

\definecolor{darkblue}{rgb}{0,0,0.7}
\definecolor{darkred}{rgb}{0.6,0,0}

\newcommand{\norm}[1]{\| #1 \|}

\newcommand{\al}{\alpha}

\newcommand{\de}{\delta}
\newcommand{\e}{\epsilon}
\newcommand{\ga}{{\gamma}}

\newcommand{\la}{\lambda}

\newcommand{\om}{{\omega}}
\newcommand{\si}{\sigma}
\newcommand{\td}{\tilde}

\newcommand{\De}{\Delta}
\renewcommand{\th}{\theta}

\newcommand{\R}{{\mathbb R }}\newcommand{\RR}{{\mathbb R }}
\newcommand{\N}{{\mathbb N}}

\newcommand{\pd}{{\partial}}
\newcommand{\nb}{{\nabla}}
\newcommand{\lec}{\lesssim}

\newcommand{\I}{\infty}
 
\renewcommand{\div}{\mathop{\mathrm{div}}}

\newcommand{\supp}{\mathop{\mathrm{supp}}}

\newcommand{\donothing}[1]{{}}

\newcommand{\EQ}[1]{\begin{equation}\begin{split} #1 \end{split}\end{equation}}
\newcommand{\EQN}[1]{\begin{equation*}\begin{split} #1 \end{split}\end{equation*}}

\DeclareMathOperator*{\esssup}{ess\,sup}

\makeatletter
\newcommand{\xRightarrow}[2][]{\ext@arrow 0359\Rightarrowfill@{#1}{#2}}
\makeatother

\newcommand{\loc}{\mathrm{loc}} 
\newcommand{\far}{\mathrm{far}} 
\newcommand{\near}{\mathrm{near}} 
\newcommand{\uloc}{\mathrm{uloc}}

\newcommand{\cD}{\mathcal{D}}

\begin{document}
\title{On the local pressure expansion for the Navier-Stokes equations} 
\author{Zachary Bradshaw and Tai-Peng Tsai}
\date{\today}
\maketitle 

\begin{abstract}  
We show that the pressure associated with a distributional solution of the Navier-Stokes equations on the whole space satisfies a local expansion defined as a distribution if and only if the solution is mild. This gives a new perspective on Lemari\'e-Rieusset's ``equivalence theorem.'' Here, the Leray projection operator composed with a gradient is defined without using the Littlewood-Paley decomposition. Prior sufficient conditions for the local expansion assumed spatial decay or estimates on the gradient and imply the considered solution is mild. An important tool is an explicit description of the bounded mean oscillation solution to a Poisson equation, which we examine in detail. As applications we include an improvement of a uniqueness criteria by the authors in Morrey spaces and revisit a proof of a regularity criteria in dynamically restricted local Morrey spaces due to Gruji\'c and Xu.
\end{abstract}

\section{Introduction}\label{sec.intro}

The Navier-Stokes equations describe the evolution of a viscous incompressible fluid's velocity field $u$ and associated scalar pressure $p$.  In particular, $u$ and $p$ are required to satisfy
\EQ{\label{eq.NSE}
&\partial_tu-\Delta u +u\cdot\nabla u+\nabla p = 0,
\\& \nabla \cdot u=0,
}
in the sense of distributions.  For our purpose, \eqref{eq.NSE} is applied on $\R^3\times (0,\I)$ and $u$ evolves from a prescribed, divergence free initial data $u_0:\R^3\to \R^3$.

In the classical paper \cite{leray}, J.~Leray constructed  global-in-time weak solutions to \eqref{eq.NSE} on $\R^4_+=\R^3\times (0,\infty)$ for any divergence free vector field $u_0\in L^2(\R^3)$.  Leray's solution $u$ satisfies the following properties:
\begin{enumerate}
\item $u\in L^\I(0,\I;L^2(\R^3))\cap L^2(0,\I;\dot H^1(\R^3))$,
\item $u$ satisfies the weak form of \eqref{eq.NSE},
\[
\iint -u \pd_t \zeta + \nb u:\nb \zeta + (u \cdot \nb )u \cdot \zeta = 0,\quad \forall \zeta \in C^\infty_c(\R^4_+;\R^3), \quad \div \zeta=0,
\]
\item $u(t)\to u_0$ in $L^2(\R^3)$ as $t\to 0^+$,
\item $u$ satisfies the \emph{global energy inequality}: For all $t>0$,
\[
\int_{\R^3} |u(x,t)|^2\,dx +2 \int_0^t \int_{\R^3} |\nabla u(x,t)|^2\,dx\,ds \leq \int_{\R^3} |u_0(x)|^2\,dx.
\]
\end{enumerate}
 The above existence result was extended to domains by Hopf in \cite{Hopf}.

In his book \cite{LR}, Lemari\'e-Rieusset introduced a local analogue of Leray-Hopf weak solutions called \emph{local Leray solutions}.  These solutions evolve from uniformly locally square integrable data  $u_0\in L^2_{\uloc}$. Here, for $1\le q \le \infty$, $L^q_{\uloc}$ is the space of functions on $\R^3$ with finite norm
\[
\norm{u_0}_{L^q_{\uloc}} :=\sup_{x \in\R^3} \norm{u_0}_{L^q(B(x,1))}<\infty.
\]
 We can also define parabolic versions of the uniform local $L^p$ spaces: $u\in L^p_\uloc(\R^3\times (0,T))$ if and only if
\[
\|u\|_{L^p_\uloc(\R^3\times (0,T))}^p:=\sup_{x_0\in \R^3} \int_0^T\int_{B_1(x_0)} |u|^p \,dx \,dt<\I.
\]
Because we work elusively with $\R^3$, we adopt the shorthand $L^p_\uloc(0,T)$ for $ L^p_\uloc(\R^3\times (0,T))$.   Note also that $f\in L^p_\loc(\R^3\times (0,T))$ means $f \in L^p(B_R\times (\e,T-\e))$ for all $R>0$ and $0<\e\ll1$ while $f\in L^p_\loc(\R^3\times [0,T])$ means  $f \in L^p(B_R\times (0,T))$ for all $R>0$.

Having a notion of weak solution in a broader class than Leray's is useful when analyzing initial data in critical spaces such as the Lebesgue space $L^3$, the Lorentz space $L^{3,\I}$, or the Morrey space $M^{2,1}$, all of which embed in $L^2_{\uloc}$ but not in $L^2$ (see \cite{JiaSverak} for an example where this was crucial).  By critical spaces we mean spaces for which the norm of $u$ is scaling invariant.  It is in such spaces that many arguments break down.  For example, $L^\I(0,T;L^3)$ is a regularity class for Leray-Hopf solutions \cite{ESS}, but this in unknown for $L^\I(0,T;L^{3,\infty})$. 

The following definition is motivated by those found in \cite{LR,KiSe,JiaSverak-minimal,JiaSverak}.   Note that, at this point, we do not include an assumption on the pressure nor on the decay of the solution.

\begin{definition}[Local Leray solutions]\label{def:localLeray} Fix $0<T<\I$. A vector field $u\in L^2_{\loc}(\R^3\times [0,T] )$ is a local Leray solution to \eqref{eq.NSE} with divergence free initial data $u_0\in L^2_{\uloc}(\R^3)$ (denoted $u\in \mathcal N(u_0)$) if:
\begin{enumerate}
\item for some $p\in L^{3/2}_{\loc}(\R^3\times [0,T])$, the pair $(u,p)$ is a distributional solution to \eqref{eq.NSE},
\item for any $R>0$, $v$ satisfies
\begin{equation}\notag
\esssup_{0\leq t<R^2 \wedge T}\,\sup_{x_0\in \R^3}\, \int_{B_R(x_0 )}\frac 1 2 |u(x,t)|^2\,dx + \sup_{x_0\in \R^3}\int_0^{R^2\wedge T}\int_{B_R(x_0)} |\nabla u(x,t)|^2\,dx \,dt<\infty,\end{equation}
\item for all compact subsets $K$ of $\R^3$ we have $u(t)\to u_0$ in $L^2(K)$ as $t\to 0^+$,
\item $u$ is suitable in the sense of \cite{CKN}, i.e., for all cylinders $Q$ compactly supported in  $ \R^3\times(0,T )$ and all non-negative $\phi\in C_c^\I (Q)$, we have  the \emph{local energy inequality}
\EQN{ 
&%
2\iint |\nabla u|^2\phi\,dx\,dt 
\\&\leq %
\iint |u|^2(\partial_t \phi + \Delta\phi )\,dx\,dt +\iint (|u|^2+2p)(u\cdot \nabla\phi)\,dx\,dt.
}
\item the function $t\mapsto \int u(x,t)\cdot w(x)\,dx$ is continuous on $[0,T)$ for any compactly supported $w\in L^2(\R^3)$.
\end{enumerate}
\end{definition}

In application, it is useful to have the following a priori bound originally due to Lemari\'e-Rieusset \cite{LR} and extended to all scales in \cite{JiaSverak-minimal}:~For all $u\in \mathcal N (u_0)$ and $r>0$ we have
\begin{equation}\label{ineq.apriorilocal}
\esssup_{0\leq t \leq \sigma r^2}\sup_{x_0\in \RR^3} \int_{B_r(x_0)}\frac {|u|^2} 2 \,dx  + \sup_{x_0\in \RR^3}\int_0^{\sigma r^2}\int_{B_r(x_0)} |\nabla u|^2\,dx\,dt <CA_0(r) ,
\end{equation}
\[
A_0(r)=rN^0_r= \sup_{x_0\in \R^3} \int_{B_r(x_0)} |u_0|^2 \,dx,
\] 
and
\begin{equation*} 
\si=\sigma(r) =c_0\, \min\big\{(N^0_r)^{-2} , 1  \big\},
\end{equation*}
for a small universal constant $c_0>0$. This bound has been used to study existence, regularity and uniqueness problems for the Navier-Stokes equations in \cite{JiaSverak,JiaSverak-minimal,LR,LR2,BT1,BT5,BT8} and other papers.

To prove this a priori bound, the pressure on the right hand side of the local energy inequality must be expressed in terms of $u$.  This is a delicate matter when $u$ is not decaying.  In \cite{JiaSverak-minimal}, this a priori bound is established provided $u$ and $p$ satisfy the following:
For any $R>0$, $x_0\in \R^3$, and $T>0$, there exists a function of time $c_{x_0,R}(t)\in L^{3/2}(0,T)$ so that, for every $0<t<T$  and $x \in B_{2R}(x_0)$  
\EQ{ \label{pressure.dec}
&p(x,t)=T_1 [(u\otimes u )(\cdot,t)\chi_{4R} (\cdot-x_0)] (x)
\\&\quad 
+ T_2 [(u\otimes u)(\cdot,t)(1-\chi_{4R}(\cdot-x_0))](x)
+ c_{x_0,R}(t),
}
where  $T_1$ is a Calderon-Zygmund operator, $T_2$ is an integral operator the kernel of which  has stronger decay at spatial infinity than $T_1$ (by an order of 1; this will be made clear in \eqref{eq.pressureexpansion}),  and $\chi_{4R} (x)$ is the characteristic function for $B_{4R}$.

The existence of parasitic solutions (see \cite{GIM, Kukavica} and comments following Theorem \ref{thrm.pressureImpliesMild}) implies this pressure formula cannot hold for every non-decaying solution.  This raises the problem of finding optimal sufficient conditions for \eqref{pressure.dec}.   Jia and \v Sver\'ak state the following sufficient condition in \cite{JiaSverak-minimal}: If $u\in \mathcal N(u_0)$ and
\EQ{\label{condition.JiaSverak}
\lim_{|x_0|\to\I} \int_0^{R^2}\int_{B_R(x_0)} |u|^2\,dx\,dt=0,\quad \forall R>0,
}
then $p$ satisfies \eqref{pressure.dec}.  That is, \eqref{pressure.dec} follows from a mild assumption on the spatial decay of the solution.  A proof of this fact is given in \cite{KMT} based on ideas in \cite{MaMiPr}. In turn, \eqref{condition.JiaSverak} can be derived from $u_0 \in E^2$, see \cite{KiSe,KMT}.  
In \cite{LR2}, Lemari\'e-Rieusset references this pressure expansion in the context of Oseen solutions. These are solutions where $u\cdot \nb u +\nb p$ is replaced by $\mathbb P \nb \cdot(u\otimes u)$  and $\mathbb P$ is defined using Littlewood-Paley. We discuss this further in Comment 2 following the statements of our main theorems.

This paper is concerned with establishing necessary and sufficient conditions for the pressure formula \eqref{pressure.dec}, or a weaker version of it, that do not assume decay at spatial infinity or require knowledge of spatial gradients.  We will do this for a larger class than local Leray solution, but emphasize local Leray solutions because they have proven useful in a number of recent papers.  We introduce some notation before formulating our main result.

Let  
\EQ{\label{Gij.def}
G_{ij}f=R_iR_jf= - \frac 13 \de_{ij} f(x) + p.v. \int K_{ij}(x-y) f(y)dy,
}
with $R_j=(-\De)^{-\frac12}\pd_j$ the $j$-th Riesz transform and
\[  K_{ij}(y) =
\pd_i \pd_j \frac1{4\pi |y|} = \frac {-\de_{ij}|y|^2 + 3 y_i y_j}{4\pi |y|^5}.
\]
The principle-value integral in \eqref{Gij.def} converges in $L^q(\R^n)$ (see \cite[page 35]{Stein2}) if $f \in L^q(\R^n)$, $q<\I$.  If $q=\I$, then \eqref{Gij.def} is understood as an element of $BMO$ via duality with $\mathcal H^1$. We elaborate on this in Section \ref{sec.pressureformula}.

Fix a function $\Theta\in C_c^\I$ so that $\Theta = 1$ on $B_2(0)$ and $\supp\Theta \subset B_4(0)$. Let $\th_R(x) = \Theta( x/ R)$ so that $\th_R =1$ on $B_{2R}(0)$ and $\supp \th_R \subset B_{4R}(0)$.
 Also let 
\[
K_{ij}^{2R}(x) = K_{ij}(x)(1-\theta_R (x)),
\]
and
\EQ{\label{GijB.def}
G_{ij}^B f(x) =   - \frac 13 \de_{ij} f(x) +\lim_{\e\to 0} \int_{ {|x-y|>\e}} (K_{ij}(x-y)-K_{ij}^{2R} ( {x_0}-y)  )  f(y)\,dy.
}
Note that, for $x\in B$, 
\EQ{ 
G_{ij}^B(f_{ij})(x)&= -\frac 1 3 \de_{ij}f_{ij} (x)+ p.v. \int (K_{ij}(x-y)-K_{ij}^{2R}(x_0-y))f_{ij}(y)\,dy 
\\&=  -\frac 1 3 \de_{ij}f_{ij} +p.v. \int K_{ij}(x-y) f_{ij}(y) \th(x_0-y)\,dy  
\\&\quad+ \int_{|x_0-y|\geq 2R} (K_{ij}(x-y)-K_{ij}^{2R}(x_0-y))f_{ij}(y)(1-\th(x_0-y)) \,dy  
\\&= [- \Delta^{-1}\div\div ]_{ij}(f_{ij}   \th_R(x_0-\cdot) ) (x)  
\\&\quad+\int  (K_{ij}(x-y)-K_{ij}(x_0-y))  (1- \th_R(x_0-y))f_{ij}(y) \,dy , 
\label{eq.pressureexpansion}
}
which matches the structure of \eqref{pressure.dec} modulo a constant.   Unlike \eqref{Gij.def}, the principle-value integral in \eqref{GijB.def} converges in $L^q(\R^n)$ if $f_{ij} \in L^\I(\R^n)$ because the argument is localized. The remaining part converges pointwise due to the extra decay of the kernel in the far-field.

\begin{definition}\label{def.localPressure}
Assume $u\in L^2_\uloc(0,T)\cap L^q_\loc(\R^3\times (0,T) )$  for some $T>0$ and $q>2$, is a distributional solution to \eqref{eq.NSE} and $p$ is the associated pressure.  We say that $p$ satisfies the {\bf local pressure expansion} (abbreviated {\bf LPE}) \eqref{pressure-exp}
if, for every $R>0$, $t\in (0,T)$ and $x_0\in \R^3$, there exists $c_{x_0,R}(t)\in L^1(0,T)$ so that %
\EQ{\label{pressure-exp}
p = G_{ij}^{B_R(x_0)} (u_iu_j  ) +c_{x_0,R}(\cdot),
}
in $L^{q/2}(B_R(x_0)\times (0,T))$.
If $u\in \mathcal N(u_0)$ for some $u_0\in L^2_\uloc$ and the associated pressure $p$ satisfies the LPE, 
then we say $u$ is a  local energy solution. 
\end{definition}

It is worth explicitly noting that the a prior {bound \eqref{ineq.apriorilocal} is} valid for local energy solutions (see \cite{BT8}).  The assumption that $u\in L^q_\loc(\R^3\times (0,T) )$  is needed to make sense of $[- \Delta^{-1}\div\div ]_{ij}( u_i u_j   \th(x_0-\cdot) ) (x) $.  Without this assumption, the most we know about $u_i u_j   \th(x_0-\cdot)$ is that it is in $L^1$ at almost every time. Since the Riesz transforms are not bounded on $L^1$, it is not obvious that $[- \Delta^{-1}\div\div ]_{ij}( u_i u_j   \th(x_0-\cdot) ) (x) $ is defined at almost every time.  To get around this, we will use a weaker notion of the local pressure expansion.

\begin{definition}\label{def.localPressure2}
Let
$u\in L^2_\uloc(0,T) $ and $p$ be a distribution on  $\R^3 \times (0,T)$ for some $T>0$ be a distributional solution to \eqref{eq.NSE}. 
We say that $p$ satisfies the {\bf distributional local pressure expansion} (abbreviated {\bf DLPE}) if, for every $x_0\in \R^3$ and $R>0$, there exists $c_{x_0,R}(t) \in L^1(0,T)$ so that for all 
$\psi\in \mathcal D(B_R(x_0)\times (0,T))$,
\EQ{\label{eq.map}
&\langle p - c_{x_0,R}(t) ,\psi \rangle = \langle  p_\near,\psi\rangle +\langle     p_\far,\psi \rangle
\\&:= \int_0^T \int u_i (x,t)u_j(x,t) \th_R(x_0-x) R_iR_j \psi (x,t)\,dx\,dt
\\&+ \int_0^T\int \int (K_{ij}(x-y) - K_{ij}(x_0-y)) (1-\th_R (x_0-y)) (u_iu_j )(y,t)\,dy \, \psi(x,t) \,dx\,dt.
}
\end{definition}

We emphasize that in Definition \ref{def.localPressure}, the pressure is a function whereas in Definition \ref{def.localPressure2}, the pressure is only defined as a distribution.

We are often interested in the distribution $\nb p = \nb p_\near +\nb p_\far$ which is defined by the mapping
\begin{align}
\psi \mapsto &\int_0^T \int u_i (x,t)u_j(x,t) \th_R(x_0-x) R_iR_j ( \nb \cdot \psi) (x,t)\,dx\,dt \nonumber
\\&+ \int_0^T\int \int (K_{ij}(x-y) - K_{ij}(x_0-y)) (1-\th_R (x_0-y)) (u_iu_j )(y,t) \,dy \nb \cdot \psi (x,t)\,dx\,dt \nonumber
\\&=: \langle \nb   p_\near,\psi\rangle +\langle \nb   p_\far,\psi \rangle,\nonumber
\end{align}
where $\th_R$ and $x_0$ are as above and $\psi \in  \mathcal D(\R^3\times (0,T))^3$.  The difference here compared to \eqref{eq.map} is that the constant has disappeared because $\nb \cdot \psi$ has mean zero.

We now turn our attention to better understanding when solutions satisfy the local pressure expansions.   Note that  $\mathbb P = (P_{ij})_{1\leq i,j\leq 3}$ where $P_{ij}=\delta_{ij}+R_iR_j$.

\begin{theorem}\label{thrm.pressureNSE}
Let $u_0\in L^1_\loc(\R^3; \R^3)$ be divergence free and satisfy $e^{t\Delta}u_0\in L^1_\uloc(0,T)$.
 Let $u:\R^3\times (0,T)\to \R^3$ be given so that
$u \in L^2_\uloc(0,T)$ 
and
\EQ{\label{condition.mild}
u(x,t)=e^{t\Delta}u_0(x) -\int_0^t e^{(t-s)\Delta} \mathbb P(\nabla \cdot (u\otimes u))(s)\,ds%
}
in $L^1_\uloc(0,T)$.
Then, there exists $\bar p\in \mathcal D'(\R^3\times (0,T))$ so that $u$ and $ \bar p$ solve the Navier-Stokes equations distributionally and $ \bar p$ satisfies the DLPE. 
Additionally:
\begin{itemize}
\item if $u\in  L^q_\loc(\R^3\times [0,T) )$ for some $q>2$, then  $\bar p\in L^{q/2}_\loc(\R^3\times [0,T))$ and satisfies the LPE;
\item  if $u_0\in L^2_\uloc$ and 
$u\in \mathcal N(u_0)$, then $u$ is a local energy solution. That is, any mild local Leray solution is a local energy solution.
\end{itemize} 
\end{theorem}

Note that we state a precise definition (from \cite{MaTe}, see also \cite{LR}) of the operator $e^{t\Delta}\mathbb P\nb\cdot$ in Section \ref{sec.mild.formulation} to ensure the integral formula above is meaningful.   

We will refer to any solution satisfying \eqref{condition.mild} as a {\bf mild solution}. On the other hand, if $u$ is a solution and there exists a  pressure $p$ so that $p$ satisfies the DLPE, we will refer to $u$ as a {\bf DLPE solution}. 
The  main assertion of Theorem \ref{thrm.pressureNSE} is that any mild solution in $L^2_\uloc (0,T)$  is a DLPE solution. Furthermore, any mild local Leray solution is a local energy solution.

The next theorem is roughly a converse to Theorem \ref{thrm.pressureNSE}.

\begin{theorem}\label{thrm.pressureImpliesMild}
Let $u_0:\R^3\to \R^3$ and $u:\R^3\times (0,T]\to \R^3$ be given divergence free vector fields so that $u_0\in L^2_\loc$, $e^{t\Delta}u_0 \in L^1_\uloc(0,T)$,  $u  \in L^2_\uloc(0,T)$
and, for every compact set $K$,
\EQN{ 
\| u(t)-u_0\|_{L^2(K)}\to 0,
}
as $t\to 0^+$.
Assume that $u$ and $p$ solve the Navier-Stokes as distributions and $p$ satisfies the DLPE. Then, $u$ satisfies the integral formula \EQ{ \notag
u(x,t)=e^{t\Delta}u_0(x) -\int_0^t e^{(t-s)\Delta} \mathbb P(\nabla \cdot (u\otimes u))(s)\,ds,
}
where equality is understood in $L^1_\loc(\R^3\times (0,T))$. It follows that every local energy solution is mild.
\end{theorem}

\bigskip
{\bf Comments on Theorems \ref{thrm.pressureNSE} and \ref{thrm.pressureImpliesMild}} 
\begin{enumerate} 

\item This paper provides a new approach to ideas in \cite{LR} and \cite{LR2}. The major difference is the definition of $\mathbb P\nb \cdot (u\otimes u)$.   In \cite[Ch.~11]{LR} and \cite[Ch.~6]{LR2}, Lemari\'e-Rieusset defines $\mathbb P \nb\cdot (u\otimes u)$ in $\mathcal D'( \R^3\times (0,T) )$ using the Littlewood-Paley decomposition in the space variable under the assumption that $u\in L^2_\uloc(0,T)$. 
He additionally shows that there exists a distribution $P$ so that $\mathbb P  \nb\cdot (u\otimes u)=\nb \cdot (u\otimes u) +\nb P$. Again, $P$ is defined using the Littlewood-Paley decomposition.  Consequently, if $u$ solves 
\EQ{\label{eq.nse.diff}
\partial_t u - \Delta u +\mathbb P \nb \cdot (u\otimes u)=0,\qquad \nb\cdot u = 0,
}
as a distribution then $u$ and $P$ also solve
\[
\partial_t u  - \Delta u + \nb \cdot (u\otimes u) +\nb P=0,\qquad \nb\cdot u = 0,
\]
as distributions. This is the content of \cite[Theorem 11.1.i]{LR}.  In \cite[Theorem 11.2]{LR}, Lemari\'e-Rieusset proves that, assuming $u\in L^2_\uloc(0,T)$, $u$ is a distributional solution to 
\eqref{eq.nse.diff}
if and only if there exists a divergence free $u_0\in \mathcal S'(\R^3)$ so that  $u$ is a distributional solution to 
\[
u = e^{t\Delta}u_0 -\int_0^t e^{(t-s)\Delta}\mathbb P \nb \cdot (u\otimes u)\,ds.
\]
Taking these results together, Lemari\'e-Rieusset roughly proved that $u\in L^2_\uloc(0,T)$  is mild if and only if there exists a distribution $P$ defined using Littlewood-Paley so that $u$ and $P$ solve the Navier-Stokes equations as distributions.

We are essentially re-proving Lemari\'e-Rieusset's theorems via a different approach which we hope will prove useful in applications. In particular, we define the pressure via the DLPE, instead of using Littlewood-Paley.  Put differently, we can define the operator  $\mathbb P \nb\cdot$ directly as 
\[
 \mathbb P\nb \cdot (u\otimes u) = u\cdot \nb u + \nb p,
\]
where $ p$ satisfies the DLPE.
The motivation for this   are applications where the local pressure expansion is more appropriate. Indeed, the local pressure expansion has been used to good effect in a variety of papers (see e.g. \cite{BT8,JiaSverak-minimal,JiaSverak,KMT} and others),  but has not been examined in the full generality of Lemari\'e-Rieusset's equivalence theorem.    Additionally, we expect that by avoiding the Littlewood-Paley approach entirely, i.e.~by not using Lemari\'e-Rieusset's existing equivalence theorem and its associated definition of the pressure, our approach will prove useful on the half-space where a similar pressure formulation has been given \cite{MaMiPr}.

Lemari\'e-Rieusset explores the local pressure expansion  \cite[Lemma 6.4]{LR2}, but requires estimates on the gradient of $u$ to makes sense of it, which we do not.  Indeed,  \cite[Lemma 6.4]{LR2} essentially corresponds to the last sentence in Theorem \ref{thrm.pressureNSE}. It does not reach the full generality of Theorem \ref{thrm.pressureNSE} nor Theorem \ref{thrm.pressureImpliesMild}.

Note that Lemari\'e-Rieusset  includes a forcing in \cite{LR2} and we do not consider this here.  Also, in Lemari\'e-Rieusset's equivalence theorem \cite[Ch. 11]{LR}, the initial data is not mentioned explicitly (some conditions are implicit because $e^{t\Delta}u_0$ must make sense). In contrast, in Theorem \ref{thrm.pressureNSE}, we   assume  the solution converges locally to the initial data.

\item The condition from \cite{JiaSverak} implies $u$ is a mild solution.   Hence Theorem \ref{thrm.pressureNSE} can be seen as a generalization of it. 

\item  In Theorem \ref{thrm.pressureImpliesMild}, the continuity assumption at $t=0$ can be replaced by the assumption that $u\in L^\I (0,T;L^2_\uloc)$.  This condition implies $u\in L^2_\uloc(0,T)$.  

\item Local Leray solutions satisfy the assumption in Theorems \ref{thrm.pressureNSE} and  \ref{thrm.pressureImpliesMild} and hence local Leray solutions are mild if and only if they are local energy solutions.

\item 
A condition along the lines of  \eqref{condition.mild} is necessary because there exist solutions to the Navier-Stokes equations for which the local pressure expansion is not valid.  Such solutions are called ``parasitic'' in the literature.  In particular, let $u(x,t)=g(t)$ where $g$ is continuous and $g(0)=0$ and $p(x,t)=x \cdot g'(t)$.  Then $u$ and $p$ meet all the conditions in Theorem \ref{thrm.pressureNSE} except \eqref{condition.mild}, but the conclusion fails.  On the other hand, the trivial solution $u=p=0$ satisfies \eqref{condition.mild} and the conclusion holds.     
 
\item In the first theorem we require $e^{t\Delta}u_0\in L^1_\uloc(0,T)$. This holds for data $u_0$ in $L^1_\uloc$ 
or
$BMO^{-1}$. We also require $u\in   L^2_\uloc(0,T)$. This is a very weak assumption and holds e.g. for   the small data mild solutions in \cite{Koch-Tataru} (see also   \cite{LR,GIM,MaTe,KuVi}).  %

\end{enumerate}

The main tool used to prove Theorem \ref{thrm.pressureNSE} is a result on the structure  of solutions to a Poisson equation belonging to $BMO$.  This result is stated in Section \ref{sec.pressureformula}. In Section \ref{sec.presure.loc} we establish several properties of the DLPE.
The precise definition of mild solutions is recalled in Section \ref{sec.mild.formulation}. To connect the Poisson equation of Section \ref{sec.pressureformula} to the Navier-Stokes equations, we approximate the nonlinearity by bounded fields.  This approximation is explored in Section \ref{sec.approx}.  
We prove Theorems \ref{thrm.pressureNSE} and \ref{thrm.pressureImpliesMild} respectively in Sections \ref{sec.proof.main}  and \ref{sec.proof.main.2}.  The main idea is that the LPE and DLPE are equivalent to the mild formulation at the level of the approximated problem, and, upon taking limits, this equivalence also holds for a given solution.  Finally we include several applications in Section \ref{sec.applications}.
 
Note that one can try to approximate the nonlinearity with bounded, compactly supported terms as in \cite{KwTs} instead of just bounded fields.  The benefit is an easier analysis of the auxiliary Poisson equation (just use the $L^2$ theory), with a more complicated approximation of the nonlinearity. We choose to emphasize the $BMO$ approach to the Poisson equation since it is interesting in its own right and simplifies bootstrapping arguments. It also justifies the pressure expansion for the mollified equation used in \cite[Ch. 32]{LR}.  Finally, several other papers which rigorously establish the local pressure expansion do so using decaying solutions as a foundation \cite{KMT,KwTs} in their approximation. Our approach shows this decay is not essential.

Regarding notation, we do not distinguish between spaces of scalars and vectors, e.g. we will say $\phi\in C_c^\I(\R^3\times (0,T) )$ even when $\phi$ is a vector and we should technically say $\phi\in (C_c^\I(\R^3\times (0,T) ))^3$.

We finally note that   P.~G.~Fern\'andez-Dalgo and P.~G.~Lemari\'e-Rieusset very recently released an elegant paper \cite{FGLR2} which addresses questions about the pressure representation of distributional solutions in a very large class based on a weighted $L^2$ space appearing in \cite{FL,FDJ}.  Our comments in the introduction do not reflect developments in \cite{FGLR2}.

\section{A Poisson equation with non-decaying source}\label{sec.pressureformula}

In the note \cite{Feff}, C.~Fefferman states that, for $g\in L^\I$, the Riesz transforms can be defined by 
\EQN{
R_j^0(g)(x)=\lim_{\delta\to 0} \int_{\delta <|x-y|} [ K_j(x-y) -K_j^0(-y) ]\, g(y)\,dy,
}
where $j=1,2,3$,
\[
K_j(y)=c \frac {y_j} {|y|^4},
\]
and 
\[
K_j^0(y)=\begin{cases}
c \frac {y_j} {|y|^4} &\text{ if }y>1
\\ 0 &\text{ if }y\leq 1.
\end{cases}
\]
Note that for $g\in L^p$, the classical Riesz transforms are defined differently (namely, without $K_j^0$).  The modification here ensures $R_j^0(g)$ makes sense when $g$ is not decaying.  It agrees with the usual definition modulo a constant when $g$ has compact support.  This representation resembles the structure of the LPE, but with obvious differences.  A related discussion can be found in \cite[p.~156-157]{Stein}, where an equivalent characterization is given using duality.  In particular, if $\psi$ is bounded, compactly supported and has mean zero (implying it is in $\mathcal H^1$), and $f\in L^\I$, then the Riesz transform of $f$ satisfies
\[
[ R_i f] (\psi)  = -\int f (R_i \psi) \,dx.
\] 
Note that, on the right hand side, the Riesz transform {is} 
defined using the classical formula.
In this section, we examine the Riesz transforms understood using duality and their relationship to the $BMO$ solution of a Poisson equation.  %
In particular we establish  the LPE directly from the duality definition of Riesz transforms and check that the properties relating to a Poisson equation of the classical Riesz transforms remain valid for the modified transforms.  
 
If $p$ is the associated pressure for a solution $u$ of the Navier-Stokes equations,  then $p$ solves the Poisson equation
\EQ{\label{eq.poisson}
-\Delta p = \partial_i \partial_j f_{ij},
}
in the distributional sense where $f=(f_{ij})=(u_iu_j)$ and we are implicitly summing over $1\leq i,j\leq 3$. A  distribution $p$ solves \eqref{eq.poisson} if and only if
\EQ{\label{eq.poisson.weakform}
\langle p , \Delta\phi\rangle =- \int f_{ij} \pd_i \pd_j \phi \,dx, \quad \forall \phi \in C^\infty_c(\R^3),
}
where $\langle\cdot,\cdot\rangle$ represents the duality pairing for $\mathcal D'$ and $\mathcal D$.  If $p\in L^1_\loc$, we let 
\[
\langle p,\psi \rangle = \int p \, \psi \,dx.
\]
The right hand side {of \eqref{eq.poisson.weakform}} can be altered using \cite[p.~59]{Stein2} which says $\partial_i\partial_j  \phi =-R_iR_j \Delta \phi$ where $R_iR_j$ are the classical Riesz transforms. In particular, this means a distribution $p$ solves \eqref{eq.poisson} if and only if
\EQ{\label{eq.pandbarp}
\langle p, \Delta \phi \rangle
=  -\int f_{ij} \partial_i\partial_j  \phi \,dx = \int f_{ij} R_iR_j \Delta \phi \,dx.
}

We are interested in solutions $u$ of \eqref{eq.NSE} which are non-decaying. The case $u\in L^\I$ is a natural foundation to begin studying such solutions.  
This leads us to consider the problem
\eqref{eq.pandbarp}
where $f_{ij}\in L^\I$ for $1\leq i,j\leq 3$.

Above we specified what a solution $p$ to the Poisson equation should do when tested against $\Delta \phi$.   We now construct a solution $p\in BMO$ to \eqref{eq.poisson} given $f_{ij}\in L^\I$ for $1\leq i,j\leq 3$.
Recall that $\mathcal H^1_a$ is the dense subspace of $\mathcal H^1$ consisting of finite linear combinations of $\mathcal H^1$ atoms.  Note that this subspace consists of all bounded functions with compact support and mean zero \cite[p.~112]{Stein}.
  Let $l$ be the bounded linear function on $\mathcal H^1$ given by
\EQN{ 
l(h)=R_i R_j (f_{ij}) (h)=\int f_{ij} R_iR_j h \,dx,\text{ for all } h\in \mathcal H^1.
}
Since $R_iR_j h$ may not be in $\mathcal H^1_a$,  the above integral may not converge in general (see \cite[p.~142]{Stein}), but it does here because $f_{ij} \in L^\I$. In particular we have 
\[
\bigg|\int f_{ij}R_iR_j h\,dx \bigg| \le \norm{f}_\infty \norm{R_iR_j h}_{L^1} \lec \norm{f}_\infty \norm{R_iR_j h}_{\mathcal H^1} \lec \norm{f}_\infty \norm{h}_{\mathcal H^1}.
\]
This confirms that $l$ is a bounded linear functional on $\mathcal H^1$.
Thus, by \cite[p.~142]{Stein}, there exists $p\in BMO$ so that
\[
\int p \,h \,dx = l(h) \text{ for all }h\in \mathcal H^1_a.
\]

Whenever $\phi \in C_c^\I$, $\Delta\phi$ has mean zero and is compactly supported and bounded.  Hence it is in $\mathcal H^1_a$.  So, for all $\phi \in C_c^\I$ and letting $h=\Delta \phi$ in the above identity, we have
\[
\int p \Delta \phi \,dx=\int f_{ij} R_iR_j \Delta \phi \,dx,
\]
i.e.~$p$ satisfies the Poisson equation in the sense of distributions.     Furthermore, note that if $\phi \in (C_0^1)^3$, then $\nb\cdot \phi$ has mean zero by the divergence theorem.  Hence $\nb\cdot \phi \in \mathcal H^1_a$ and, therefore,
\[
\int p\nb \cdot \phi \,dx=\int f_{ij}R_iR_j  \nb \cdot \phi \,dx.
\]

The following theorem elaborates on the above formula.  In particular, it gives an explicit formula for solutions to \eqref{eq.poisson} and allows us to move the integral operators from $\nb\cdot \phi$ to $f_{ij}$.

\begin{theorem} \label{thrm.pressure}
Assume that $p\in BMO$ solves the Poisson equation
\EQ{\label{eq.poisson-2} 
\Delta p = -\partial_i\partial_j f_{ij},
}
in the distributional sense where $f_{ij}\in L^\I\cap C^\al_\loc$, for some $\al>0$ and $1\leq i,j\leq 3$ (we implicitly sum over $i$ and $j$).

Define $\bar p$  as follows: 
\begin{itemize}
\item If $x\in B_1(0)$ then $\bar p (x)= G_{ij}^{B_2(0)} f_{ij}(x)$.
\item If $x\in B_{n}(0)$ then $\bar p (x)= G_{ij}^{B_{n+1}(0)}f_{ij} (x)+ \sum_{k=2}^n c_k$, where  
\[
c_k =     G_{ij}^{B_{k}(0)}f_{ij}-G_{ij}^{B_{k+1}(0)}f_{ij}=   {-}\int \big( K_{ij}^{2k} (-y)   - K_{ij}^{2(k+1)}(-y)\big) f_{ij} (y)\,dy,
\]
is a constant. %
\end{itemize}
Then, $\bar p$ is a well defined  distributional solution to \eqref{eq.poisson-2} and, for any ball $B$,
\EQ{\label{eq.pressureformula}
\int_{\R^3} p \,  \psi  \,dx = \int_{\R^3} (G_{ij}^Bf_{ij})\psi\,dx=\int_{\R^3}\bar p  \,\psi\,dx ,
}
for all $\psi$ which are bounded,  supported in $B$, and have mean zero (i.e. $\psi\in \mathcal H^1_a$).  This implies $p=\bar p$ in $BMO$.

\end{theorem}

\noindent {\bf Comments on Theorem \ref{thrm.pressure}}.
\begin{enumerate}

\item By construction, $\bar p \in L^q_\loc$ for any $q<\infty$ using $f_{ij}\in L^\infty$. It is indeed in $L^\infty_\loc$ using $f_{ij}\in  L^\infty\cap C^\al_\loc$.

\item We will apply this theorem to smooth approximations and hence the $C^\al_\loc$ assumption is not an obstacle.

\item Assuming %
a distributional solution $p$ of \eqref{eq.poisson-2} is in $BMO$ is equivalent to assuming $p=R_iR_j f_{ij}$.  This is because $R_iR_j f_{ij}$ is in $BMO$ and is also a solution to the Poisson equation. Uniqueness in $BMO$ implies $p=R_iR_j f_{ij}$ \cite{Kukavica}. That $\bar p\in BMO$ follows from the fact that $p$ and $\bar p$ behave the same as functionals on a dense subset of $\mathcal H^1$.  This implies that $R_iR_j f_{ij}=\bar p$ in BMO.

\item 
In Theorem \ref{thrm.pressure}, we are checking that the dual characterization of $p$ agrees in $BMO$ with the explicit formula for $\bar p$.   Understanding $p\in BMO$ when $u\in L^\I$ is classical; see for example \cite{Kukavica,KuVi,GIM, Tsai-ARMA}.  
We include a detailed proof under general conditions and explicitly check that the properties of classical Riesz transforms extend to the modified formulas because we cannot find these details concisely contained in the literature.

\item We may replace the sum appearing in the definition of $\bar p$ by
\[
\td c_n =  \sum_{k=2}^n c_k =  G_{ij}^{B_{2}(0)}f_{ij}-G_{ij}^{B_{n+1}(0)}f_{ij}= {-}\int \big( K_{ij}^{2} (-y)   - K_{ij}^{2(n+1)}(-y)\big) f_{ij} (y)\,dy.
\]
Clearly $\td c_n \lesssim \ln (n)$ whenever $f_{ij}$ is in $L^\I$ because
\EQN{
G_{ij}^{B_{{2}}(0)}f_{ij}-G_{ij}^{B_{n+1}(0)}f_{ij}\leq C \int_{1\leq |y|\leq 4n+4}  \frac 1 {|y|^3}\,dy \lesssim \ln n,
}
as $n\to \I$. 
This is consistent with $\bar p \in BMO$.
\end{enumerate}

 The proof of Theorem \ref{thrm.pressure} is conceptually straightforward: Expand the integral operator and use Fubini's theorem to move the operator from the test function to $f_{ij}$.  There are, however,  several applications of the dominated convergence theorem, the justification for which gets rather technical.  We check that the premises of the dominated convergence theorem are satisfied in the following lemma.  
 
\begin{lemma}\label{lemma.forDCT}Fix $x_0\in \R^3$ and $R>0$.  
For $1\leq i,j\leq 3$, assume that $f_{ij}\in L^\I\cap C^\al_{\loc}$ and $\psi\in C^\al_{\loc}$ for some $\al>0$, $\supp \psi \subset B_R(x_0)$, and $\int \psi\,dx = 0$.  Then, there exist  non-negative integrable functions $F_1,\,F_2,$ so that
\begin{enumerate}
\item  $|F_1^N(x)|\leq F_1(x)$ for all $x\in \R^3$ and $N>0$, where  
\[  F_1^N(x)=  f_{ij}(x) \int_{N^{-1}< |x-y| }  (K_{ij} (x-y) - K_{ij}^{2R}(x-x_0)) \psi (y)\,dy,
\]
for any $1\leq i,j\leq 3$ (in other words, there exists a single function $F_1(x)$ so that, for any choice of $i$ and $j$, $F_1^N(x)$ is bounded).
\item $|F_2^{N,M}(y)|\leq F_2(y)$ for all $y\in \R^3$, $N\in (0,\I]$ and $M\geq 2R$, where, for $N<\I$
\[
F_2^{N,M}(y) = \psi(y)  \int_{N^{-1}<|x-y| ;|x-x_0|<M}    \big(K_{ij}(x-y)-K_{ij}^{2R}(x-x_0)\big)f_{ij}(x)\,dx,
\] 
and for $N=\I$
\[
F_2^{N,M}(y) = \psi(y)  \lim_{\underline{N}\to \I} \int_{{\underline{N}}^{-1}<|x-y| ;|x-x_0|<M}    \big(K_{ij}(x-y)-K_{ij}^{2R}(x-x_0)\big)f_{ij}(x)\,dx,
\]
for any $1\leq i,j\leq 3$ (in other words, there exists a single function $F_2(y)$ so that, for any choice of $i$ and $j$, $F_2^{N,M}(y)$ is bounded).
\end{enumerate}
\end{lemma}

The functions $F_1^N, F_1, F_2^{N,M}$ and $F_2$ are everywhere  defined because $f_{ij},\psi \in C^\al_\loc$. The definitions of $F_1$ and $F_2$ depend on $\norm{f}_{C^\al(B_{2R}(x_0))}$ and $\norm{\psi}_{C^\al}$.
Since the proof of Lemma \ref{lemma.forDCT} is technical, we postpone it until after the proof of Theorem \ref{thrm.pressure}.

\begin{proof}[Proof of Theorem \ref{thrm.pressure}]

We first show $\bar p$ is well defined.  Fix $x$ and let $n$ be the smallest natural number so that $x\in B_n(0)$. On one hand we have defined
\[
\bar p(x)= G_{ij}^{B_{n+1}(0)}f_{ij}(x) +\sum_{k=2}^n c_k.
\]
For $\bar p$ to be well defined this must agree with 
\[
G_{ij}^{B_{m+1}(0)}f_{ij}(x) +\sum_{k=2}^{m} c_k.
\]
for all $m>n$.  Letting $m=n+1$ we have 
\EQN{
G_{ij}^{B_{n+2}(0)}f_{ij}(x) +\sum_{k=2}^{n+1} c_k  -G_{ij}^{B_{n+1}(0)}f_{ij}(x) -\sum_{k=2}^n c_k
&= 0,
}
by the definition of $c_{n+1}$.
For induction assume for some $m>n$ that
\EQN{
G_{ij}^{B_{m+1}(0)}f_{ij}(x) +\sum_{k=2}^m c_k= G_{ij}^{B_{n+1}(0)}f_{ij}(x) +\sum_{k=2}^n c_k.
}
Then
\EQN{
&G_{ij}^{B_{m+2}(0)}f_{ij}(x) +\sum_{k=2}^{m+1} c_k -   G_{ij}^{B_{n+1}(0)}f_{ij}(x) +\sum_{k=2}^n c_k
\\&=G_{ij}^{B_{m+2}(0)}f_{ij}(x) +\sum_{k=2}^{m+1} c_k-   G_{ij}^{B_{m+1}(0)}f_{ij}(x) +\sum_{k=2}^m c_k
=0,
}
implying $\bar p$ is unambiguously defined.

Next, assume that $\psi:\R^3\to\R$ with $\supp \psi \subset B=B_R(0)$, $\psi\in C^{\al}_\loc$, and $\int \psi \,dx=0$.  We will first prove 
\[\int_{\R^3} p \,  \psi  \,dx = \int_{\R^3} (G_{ij}^Bf_{ij})\psi\,dx=\int_{\R^3}\bar p  \,\psi\,dx,\]
under these assumptions on $\psi$ and then extend this to  $\psi\in \mathcal H^1_a$.

Since $\psi$ is supported in $B$, $G_{ij}^B (\psi) =  G_{ij}  ( 
\psi)$. %
This is obvious given that $K_{ij}^{2R}(-y)=0$ whenever $\psi(y)\neq 0$.

Note that $\psi\in \mathcal H^1_a\subset \mathcal H^1$ because it is bounded, compactly supported and has mean zero.  Hence
$G_{ij}  (\psi)\in \mathcal H^1\subset L^1$.  Due to this, and the fact that $f_{ij}\in L^\I$, the integral
\[
 \int f_{ij} G_{ij} \psi\,dx
\]
converges absolutely by H\"older's inequality. 
We also claim that
\[
\bigg|\int  G_{ij}^B( f_{ij}) \psi \,dx\bigg| <\I. 
\]
Indeed, we have  
\EQN{
&\int  G_{ij}^B( f_{ij}) \psi \,dx  
\\&=  -\frac {\delta_{ij}}{3}\int f_{ij}\psi\,dx  +\int \bigg( \lim_{\e\to 0}    \int_{  |x-y|>\e} K_{ij}(x-y) f_{ij}(y) \chi_{B_{2R}(0)}(y)\,dy   \bigg) \psi (x)\,dx 
\\&+
\int \bigg( \int (K_{ij}(x-y)-K_{ij}^{2R} (-y)  )  f_{ij}(y)  (1-\chi_{B_{ 2R}(0)}(y))\,dy \bigg) \psi (x)\,dx.
}
The first term is finite by H\"older's inequality.
Since $f \chi_{B_{2R}(0)}$ is in $L^\I$, the  singular integral term above is in $BMO$.  Furthermore, $\psi\in \mathcal H^1_a$ implying 
\EQN{
\bigg|  \int \bigg( \lim_{\e\to 0}    \int_{  |x-y|>\e} K_{ij}(x-y) f_{ij}(y) \chi_{B_{2R}(0)}(y)\,dy   \bigg) \psi (x)\,dx \bigg| &\leq C\|\psi\|_{\mathcal H^1}\| f\|_{L^\I}.
}
For the last term note that, by the extra decay of the kernel when $y\notin B_{2R}(0)$ and $x\in B_R(0)$,
\[
\bigg|\int \bigg(  \int(K_{ij}(x-y)-K_{ij}^{2R} (-y)  )  f(y)  (1-\chi_{B_{ 2R}(0)}(y))\,dy \bigg) \psi (x)\,dx \bigg|
\leq C\norm{\psi}_{L^1} \|f\|_{L^\I}.
\]

We now show that, for fixed $i,j$ (not summed),
\EQN{%
\int f_{ij} G_{ij}  \psi \,dx   =  \int  G_{ij}^B( f_{ij})  \psi \,dx .  
}
This statement is equivalent to 
\EQ{\label{Gij.id}
\int f_{ij} \tilde G_{ij} \psi \,dx   =  \int  \tilde G_{ij}^B( f_{ij}) \psi\,dx ,
}
where $\tilde G_{ij}$ and $\tilde G_{ij}^B$ are the integral parts of $G_{ij}$ and $G_{ij}^B$, respectively, 
$ G_{ij} f = \tilde G_{ij} f + \frac 13 \de_{ij} f$, and $ G_{ij}^B f =\tilde G_{ij}^B f + \frac 13 \de_{ij} f$.
The indices $i,j$ are \emph{not} summed.   

Using this definition we have
\[
\int| f_{ij} \td G_{ij} \psi |\,dx <\I,
\]
since this is true for $\int  |f_{ij}  G_{ij} \psi |\,dx $, and $\frac 1 3 \int |f_{ij} \delta_{ij} \psi |\,dx $ is obviously finite.
Hence,
\[
\int f_{ij} \td G_{ij}\psi\,dx=\lim_{M\to \I} \int_{ |x-0| <M} f_{ij} \td G_{ij} \psi\,dx.
\]

Note that for any $N,M>0$, 
\[
\int_{|x|<M} f_{ij}(x)K_{ij}^{2R} (x) \int_{N^{-1}<|x-y|}\psi(y)\,dy\,dx <\I,
\]
because the regions of integration and the integrands are all bounded.
Furthermore,  for any $M>0$, because $\psi$ is integrable and mean zero, 
\EQN{\label{eq.meanzero}
&\int_{|x|<M} f_{ij}(x) \lim_{N\to\I}\int_{N^{-1}<|x-y| } K_{ij}^{2R}(x) \psi (y)\,dy\,dx
\\&= \int_{|x|<M} f_{ij}(x)K_{ij}^{2R}(x) \lim_{N\to\I}\int_{N^{-1}<|x-y| } \psi (y)\,dy\,dx
\\&= \int_{|x|<M} f_{ij}(x)K_{ij}^{2R}(x) \int  \psi (y)\,dy\,dx
=0.
}

We are ready to establish  \eqref{Gij.id}.  
We have
\[
  \int   f_{ij}(x) {\tilde G}_{ij} \psi (x)\,dx ={\lim_{M\to \I} I_M},
  \quad {I_M} =   \int_{|x|<M}   f_{ij}(x) {\tilde G}_{ij} \psi (x)\,dx.
\]
Using the definition of {$\td G_{ij}$},
\EQN{
  {I_M}=    \int_{|x|<M}   f_{ij}(x) \lim_{N\to \I}\int_{N^{-1}< |x-y|} K_{ij} (x-y) \psi (y)\,dy\,dx .
}
The {last} limit is valid in $L^q$ and also pointwise since $\psi\in C^\al$. By equation \eqref{eq.meanzero},
 the dominated convergence theorem with limit in $N\to\I$ (noting the integrability and pointwise bounds established in part 1 of Lemma \ref{lemma.forDCT} with $f_{ij},\psi\in C^\al_\loc$), we have
\EQN{
{I_M} 
&=     \int_{ |x|<M}  f_{ij}(x)\lim_{N\to \I}\int_{N^{-1}< |x-y|}  (K_{ij} (x-y) - K_{ij}^{2R}(x)) \psi (y)\,dy\,dx 
\\&= \lim_{N\to \I}    \int_{ |x|<M} \int_{N^{-1}< |x-y|} f_{ij}(x) (K_{ij} (x-y) - K_{ij}^{2R}(x)) \psi (y)\,dy\,dx
\\&= \lim_{N\to \I}    \int  \int  \chi_{B_M(x)}(x)  \chi_{\R^3\setminus B_{N^{-1}}(x)}(y)  f_{ij}(x) 
\\&\qquad\qquad\qquad \times (K_{ij} (x-y) - K_{ij}^{2R}(x)) \psi (y)\,dy\,dx.
}     
For $M$ and $N$ fixed,
\EQN{
 & |\chi_{B_M(0)}(x)  \chi_{\R^3\setminus B_{N^{-1}}(x)}(y)  f_{ij}(x) (K_{ij} (x-y) - K_{ij}^{2R}(x)) \psi (y)|
 \\&\leq  \|f_{ij}\|_{L^\I} \chi_{B_M(0)}(x)  |\psi (y)| \big(  \|K_{ij}  \chi_{\R^3\setminus B_{N^{-1} }(0)}   \|_{L^\I} +\|K_{ij}^{2R}\|_{L^\I} \big).
}
The product of two measurable functions each depending on only one variable is measurable in the product measure.
Since the support of $K_{ij}  \chi_{\R^3\setminus B_{N^{-1} }(0)} $ is bounded away from zero and $ \chi_{B_M(0)}(x) $ and $ |\psi_{0}(y)|$ are bounded and compactly supported, the above function is integrable in the product measure. 
We may therefore use  Fubini's theorem to obtain
\EQN{
{I_M}&= \lim_{N \to \I} \bigg[ \int  \psi (y)  \int \chi_{B_M(0)}(x)  \chi_{\R^3\setminus B_{N^{-1}}(y)}(x) 
 \\&\qquad\qquad\qquad \times f_{ij}(x) (K_{ij} (x-y) - K_{ij}^{2R}(x)) \,dx\,dy\bigg].
}

Considering the above limits, part 2 of Lemma \ref{lemma.forDCT} and the dominated convergence theorem allow us to pass the limit in $N$ through the outside integral.  Hence,
\EQN{
 {I_M}=  \int  \psi (y)  \lim_{N \to \I} \int_{|x-y|>N^{-1}} \chi_{B_M(0)}(x)  
 f_{ij}(x) (K_{ij} (x-y) - K_{ij}^{2R}(x)) \,dx\,dy .
} 
 We furthermore have 
\EQN{
&{ \int   f_{ij}(x) {\tilde G}_{ij} \psi (x)\,dx = \lim_{M \to \I}I_M}
\\&=\int  \psi (y) \lim_{M\to \I} \lim_{N \to \I} \int_{|x-y|>N^{-1}} \chi_{B_M(0)}(x)  
 f_{ij}(x) (K_{ij} (x-y) - K_{ij}^{2R}(x)) \,dx\,dy
}
where we used the dominated convergence theorem and Lemma \ref{lemma.forDCT} (part 2 with $N=\I$) to move the limit in $M$ through the outside integral.  To conclude we must show that, for almost every $y$,
\EQ{\label{eq.limitM} 
&\lim_{M\to \I} \lim_{N \to \I} \int_{|x-y|>N^{-1}} \chi_{B_M(0)}(x)    f_{ij}(x) (K_{ij} (x-y) - K_{ij}^{2R}(x)) \,dx 
\\&= \lim_{N \to \I} \int_{|x-y|>N^{-1}}   f_{ij}(x) (K_{ij} (x-y) - K_{ij}^{2R}(x)) \,dx
 = \td {G}_{ij}^B f_{ij} (y).
}
Note that for $M$ fixed and $M>2R$, the limits in $N$ are both convergent pointwise, which follows because the singularities of the inside integrals are depleted by H\"older continuity while the tails are integrable.  
We can therefore take the difference of the terms in \eqref{eq.limitM} to obtain, for almost every $y$,
\EQN{
&\bigg|\lim_{N\to \I} \int_{|x-y|>N^{-1}} (\chi_{B_M(0)}(x) -1 )  f_{ij}(x) (K_{ij} (x-y) - K_{ij}^{2R}(x)) \,dx\bigg|
\\&\leq CR \|f_{ij}\|_{L^\I} \int (\chi_{B_M(0)}(x) -1 ) \frac 1 {|x|^4}  \,dx,
}
where we used the fact that $y\in B_R(0)$ and $|x|>2R$.
This clearly vanishes as $M\to \I$, and so \eqref{eq.limitM} holds for all $y$.  We have shown \eqref{Gij.id} for fixed $i,j$.

This proves that, when $\psi:\R^3\to\R$ with $\supp \psi \subset B_R(0)$, $\psi\in C^{\al}_\loc$, and $\int \psi \,dx=0$, we have  (now suming in $i,j$)
\EQN{ 
\int p \psi\,dx = \int   f_{ij}(x) { G}_{ij} \psi  (x)\,dx
 = \int{G}_{ij}^B f_{ij} (y)  \psi (y) \,dy, %
}
where the first equality is from the definition of $p\in BMO$.  

Finally, let $n$ be large enough that $B\subset B_n(0)$.  Then, 
\[
{G}_{ij}^B f_{ij}= \bar p (x)+c_B,
\]
for a constant $c_B$ depending on $B$ (this follows from a straightforward calculation).  Since $\psi$ has mean zero, we obtain
\EQN{ 
\int p \psi\,dx = \int{G}_{ij}^B f_{ij} (y)  \psi (y) \,dy= \int ( \bar p+c_B) \psi \,dx = \int   \bar p \psi \,dx .
}

We next prove \eqref{eq.pressureformula} for all $\psi\in \mathcal H^1_a$.  By \cite[Ch.~IV, Section 1.2, Theorem 1]{Stein}, this will imply that $p=\bar p$ in $BMO$. Let $\eta$ be a mollifier and let $\psi \in \mathcal H^1_a$ be given.  Let $\psi_\e = \e^{-3}\eta(\cdot/\e)*\psi$.  Then,  $\psi_\e \to \psi$ in $\mathcal H^1$.  Furthermore, $\psi_\e$ is bounded, has mean zero and is compactly supported and locally H\"older continuous.  Thus 
\[
\int p\, \psi_\e \,dx = \int \bar p\, \psi_\e \,dx.
\]
Because $f_{ij}\in L^\I\cap C^\al_{\loc}$, it is easy to see that $\bar p\in L^\I_\loc$ (just use H\"older continuity to deplete the singular part  uniformly on a given compact set and note that the far-field part of the integral is bounded by a constant depending on the given compact set and the bound for $f_{ij}$).     Let $n\in \mathbb N$ be such that $\supp \psi ,\supp \psi_\e\subset B_n(0)$
 for $\e<1$.  Then, $p-\bar p \chi_{B_{n}(0)}\in BMO$.  Thus
\EQN{
\bigg|\int (p-\bar p) \psi \,dx \bigg|&=\bigg|\int (p-\bar p) (\psi - \psi_\e)\,dx\bigg|
\\&\leq \|p-\bar p \chi_{B_{n}(0)}\|_{BMO}\| \psi-\psi_\e\|_{\mathcal H^1},
}
where we note that, while the above estimates are not true for general $\mathcal H^1$-$BMO$ pairings,  they do hold when $\psi,\psi_\e\in \mathcal H^1_a$  (see \cite[p.142]{Stein}), which is the case.
Since the right hand side of the above vanishes as $\e\to 0$, we conclude that 
\[
\int (p-\bar p) \psi \,dx=0,
\]
for all $\psi\in \mathcal H^1_a$.
Since $\mathcal H^1_a$ is dense in $\mathcal H^1$, this implies $p=\bar p$ in $BMO$. 
 This finishes the proof of Theorem \ref{thrm.pressure}.
\end{proof}

We now return to the proof of Lemma \ref{lemma.forDCT}.

\begin{proof}[Proof of Lemma \ref{lemma.forDCT}] 
\emph{1.}  We will show that for all $N>0$, the function 
\[
F_1^N(x)= f_{ij}(x) \int_{N^{-1}< |x-y|}  (K_{ij} (x-y) - K_{ij}^{2R}(x-x_0)) \psi (y)\,dy
\]
is dominated by an integrable function. Let $\chi_1(x) = \chi_{B_{2R}(x_0)}(x)$ and $\chi_2(x)=1-\chi_1(x)$. Note that, if $x\in B_{2R}(x_0)$, then $K_{ij}^{2R}(x-x_0)=0$. Hence we have
\EQN{
F_1^N(x)&=(\chi_1+\chi_2)(x) f_{ij}(x) \int_{N^{-1}< |x-y|}  \bigg\{ \cdots \bigg\}\,dy =: I_1+I_2,
\\
I_1(x) &=\chi_1(x) f_{ij}(x) \int_{N^{-1}< |x-y|}   K_{ij} (x-y)  \psi (y)\,dy ,
\\
I_2(x) &=\chi_2(x) f_{ij}(x) \int_{N^{-1}< |x-y|}   (K_{ij} (x-y) - K_{ij}^{2R}(x-x_0)) \psi (y)\,dy .
}

We first consider $I_1$. For all $x\in B_{2R}(x_0)$ we have 
\EQN{
&\bigg| \int_{N^{-1}< |x-y|}  K_{ij} (x-y)   \psi (y)\,dy \bigg|
= \bigg| \int_{N^{-1}< |x-y|<3R}  K_{ij} (x-y)   \psi (y)\,dy \bigg|
\\
&= \bigg| \int_{N^{-1}< |x-y|<3R}  K_{ij} (x-y)   [\psi (y)-\psi (x)]\,dy \bigg|
\\&\leq C\|\psi \|_{C^\al(B_{3R}(x_0))} \int_{B_{3R}(x)} \frac 1 {|x-y|^{3-\al}}\,dy\leq C {R^\al}\|\psi \|_{C^\al(B_{3R}(x_0))},
}
where we used the fact that the integral of $K_{ij}$ over any sphere centered at the origin is zero, and the assumption $\psi\in C^\al_\loc(\R^3)$. Thus
\[
|I_1(x)| \le  C {R^\al}\norm{f}_{L^\infty} \|\psi \|_{C^\al(B_{3R}(x_0))} \cdot\chi_1(x).
\]

For $I_2$, if $x\notin B_{2R}(x_0)$ and $y\in B_{R}(x_0)$, then 
\EQ{ \label{ineq.extradecay}
|K_{ij} (x-y) - K_{ij}^{2R}(x-x_0)|\leq  \frac {{CR}}  {|x_0-x|^4},
}
which can be proved directly for $2R<|x-x_0|<4R$ and by mean value theorem for $|x-x_0|\ge 4R$.
Thus,
\[
|I_2(x)| \le  \norm{f}_{L^\infty}\chi_2(x) \int_{B_R(x_0)} \frac  {CR|\psi(y)|}  {|x_0-x|^4}\,dy
\le C R^4\norm{f}_{L^\infty} \|\psi \|_{L^\infty} \frac{\chi_2(x)}{|x_0-x|^4}.
\]
We conclude
\EQN{
|F_1^N(x)|&\le |I_1|+|I_2| \le F_1(x),
\\
F_1(x)&=
CR^\al\norm{f}_{L^\infty} \|\psi \|_{C^\al(B_{3R}(x_0))}\cdot\chi_1(x)
+ C R^4\norm{f}_{L^\infty} \|\psi \|_{L^\infty} \frac{\chi_2(x)}{|x_0-x|^4}.
}
Since $F_1$ is clearly integrable and does not depend on $N$, we are done with the first part.

\medskip 

\emph{2.} For any $M\in [2R,\I)$ and $N\in (0,\I]$ ($N=\I$ is allowed),  we will show that, as a function of $y$,
\[
 F_2^{N,M}(y)= \psi(y)  \int_{N^{-1}<|x-y|;|x-x_0|<M}    \big(K_{ij}(x-y)-K_{ij}^{2R}(x-x_0)\big)f_{ij}(x)\,dx 
\]
is dominated by an integrable function. Decompose its integration domain to 2 regions: $|x-x_0|<2R$ in region 1 and $2R\le|x-x_0|\le M$ in region 2. Since $\psi(y) K_{ij}^{2R}(x-x_0)$ vanishes in region 1, we have
\EQN{
F_2^{N,M}(y)&=I_1+I_2,
\\
I_1(y)&= \psi(y)  \int_{N^{-1}<|x-y|;|x-x_0|<2R}    \chi_{B_{2R}(x_0)}(x) K_{ij}(x-y)f_{ij}(x)\,dx ,
\\
I_2(y)&=\psi(y) \int_{N^{-1}<|x-y|;2R<|x-x_0|<M} \big(K_{ij}(x-y)-K_{ij}^{2R}(x-x_0)\big)f_{ij}(x)\,dx .
}
For $I_2$, by using  \eqref{ineq.extradecay}
\[
|I_2(y)|\le C  |\psi (y)|\|f_{ij}\|_{L^\I}  \int_{N^{-1}<|x-y|;2R<|x-x_0|<M} \frac R{|x_0-x|^4}\,dx  \le C \|f_{ij}\|_{L^\I}   |\psi (y)|,
\]
with $C$ independent of $N$ and $M$.
For $I_1$, note that
\EQN{
|I_1(y)| &\leq \bigg| \psi(y) \int_{N^{-1}<|x-y|< R } K_{ij}(x-y) \chi_{B_{2R}(x_0)}(x)  f_{ij}(x)\,dx  \bigg|
\\&\qquad+ \bigg| \psi(y) \int_{N^{-1}<|x-y| ; |x-y|\geq R } K_{ij}(x-y) \chi_{B_{2R}(x_0)}(x)  f_{ij}(x)\,dx  \bigg| =: I_{1a}+I_{1b}.
}
When $y\in B_R(x_0)$ and $|x-y|<R$, then $|x-x_0|<2R$ and so $\chi_{B_{2R}(x_0)}(x)=1$. Thus 
\[
I_{1a} = \bigg| \psi(y) \int_{N^{-1}<|x-y|< R } K_{ij}(x-y)  ( f_{ij}(x)  - f_{ij}(y)  )\,dx  \bigg|.
\]
Since $f_{ij}\in C^\al_{\loc} $, we have
\[
I_{1a}\leq |\psi(y)|\|f_{ij}\|_{C^\al(B_{2R}(x_0))}  \int_{B_{R}(y)} \frac 1 {|x-y|^{3-\al}}\,dx
=CR^\al |\psi(y)|\|f_{ij}\|_{C^\al(B_{2R}(x_0))}.
\]
On the other hand, if $y\in B_R(x_0)$ and $x\in B_{2R}(x_0)$, then $|x-y|\leq 3R$.  Hence,
\EQN{
I_{1b}\leq C |\psi(y)|   \|f_{ij}\|_{L^\I}  \int_{B_{3R}(y)\setminus B_{R}(y)} \frac 1 {|x-y|^{3 }}\,dx \leq C|\psi(y)| \|f_{ij}\|_{L^\I}.
}
Combining these estimates
\EQN{
|F_2^{N,M}(y)| &\le I_{1a}+I_{1b}+|I_2| \le F_2(y),
\\
F_2(y)&=C   (  \|f\|_{L^\I}  
+   \|f\|_{C^\al(B_{2R}(x_0))}  )\, |\psi(y)|,
}
and $F_2(y)$ is clearly integrable.  
 We have shown the correct bound for $N<\I$ and $M<\I$. 

We still need to treat the case when $N=\I$. Note that $F_2^{N,M}(y) \to F_2^{\I,M}(y)$ everywhere because, for any $y\in \R^3$ and $1\ll N_1<N_2$,
\EQN{
&|F^{N_1,M}_2 - F^{N_2,M}_2|(y)
\\&\leq \bigg| \psi(y) \int_{N_2^{-1}<|x-y|<N_1^{-1}} K_{ij}(x-y)( (f_{ij})(x)-(f_{ij})(y)) \,dx\bigg|
\\&\leq \|\psi\|_{L^\I}   \|f_{ij}\|_{C^\al(B_1(y))}   \int_{N_2^{-1}<|x-y|<N_1^{-1}} \frac 1 {|x-y|^{3-\al}} \,dx.
}
The above can be made arbitrarily small by taking $N_1$ sufficiently large.  Since this is Cauchy, $F_2^{N,M}(y)$ converges  to $F_2^{\I,M}(y)$.  We know that $|F_2^{N,M}(y)|\leq F_2(y)$ for   every $y$ and since $F_2^{N,M}\to F_2^{\I,M}$ everywhere, it follows that $|F_2^{\I,M}|\leq F_2$ everywhere as well.
\end{proof}

The following lemma  ensures that, for the functions considered in Theorem \ref{thrm.pressure}, the LPE and DLPE are equivalent.

\begin{lemma}\label{lemma.adjoint}
Let $f$ be as in Theorem \ref{thrm.pressure}.  Let $\th\in C_c^\I$ be given. Then
\EQN{  
\int_{\R^3}  f_{ij }\th \, R_iR_j  \psi  \,dx = \int_{\R^3}  R_i R_j(f_{ij}\th) \,\psi\,dx ,
}
for all $\psi \in C_c^\I$.
\end{lemma}

\begin{proof}
Clearly $f_{ij}\th \in L^2$. The result then follows from the skew adjointness of Riesz transforms in $L^2$.
\end{proof}

\section{The distributional  local pressure expansion}\label{sec.presure.loc}

We first show that Definition \ref{def.localPressure2} is meaningful. 
Fix a function $\Theta\in C_c^\I$ so that $\Theta = 1$ on $B_2(0)$ and $\supp \Theta \subset B_4(0)$. Let $\th_R(x) = \Theta( x/ (R+1))$ so that $\th_R =1$ on $B_{2R+2}(0)$ and $\supp \th_R \subset B_{4R+4}(0)$.

\begin{lemma}\label{lemma.pressure.distribution}
Let
$u\in L^2_\uloc(0,T) $ for some $T>0$.  Fix $R>0$ and $x_0\in \R^3$ and let $B=B_R(x_0)$.
Consider the mapping for $\psi\in {\mathcal D}(B\times (0,T))$, 
\EQN{ 
\psi \mapsto &\int_0^T \int u_i (x,t)u_j(x,t) \th_R(x_0-x) R_iR_j    \psi  (x,t)\,dx\,dt
\\&+ \int_0^T\int \int (K_{ij}(x-y) - K_{ij}(x_0-y)) (1-\th_R (x_0-y))( u_iu_j)(y,t) \,dy\,   \psi(x,t) \,dx\,dt
\\&=: \langle   \bar p_\near^B,\psi\rangle +\langle   \bar p_\far^B,\psi \rangle.
}
Then $ \bar p^B := \bar p_\near^B + \bar p_\far^B \in \mathcal D'(B\times (0,T))$. 
\end{lemma}

\begin{proof}
We first check that the mapping is {finitely} valued.  The second term is fine by known estimates for the local pressure expansion (e.g. \cite{KiSe}). For the first, since
$
R_iR_j \psi = -\frac13\de_{ij}\psi + p.v. \int K_{ij}(\cdot -y)\psi(y)\,dy
$
(see \eqref{Gij.def})
and the kernel $K_{ij}$
 has mean zero on spheres, we have
\EQN{ \label{ineq.riesz.bounded}
|R_iR_j { \psi}(x) |&=  \bigg|-\frac13\de_{ij}\psi + p.v.\int K_{ij}(x-y)   \psi(y)\,dy   \bigg|
\\&\leq  {|\psi(x)|} + \bigg| p.v.\int_{|x-y|<1} K_{ij}(x-y) (  \psi(y)  -  \psi(x ))\,dy   \bigg| 
\\&\qquad + \bigg|  \int_{|x-y|\geq 1} K_{ij}(x-y)    \psi(y)\,dy   \bigg|
\\&\leq {\| \psi\|_{L^\infty}}+ \int_{|x-y|<1} \frac  {C \| \nb  \psi \|_{L^\I}} {|x-y|^2}dy   + \bigg( \int_{|x-y|\geq 1} \frac 1 {|x-y|^{3p'}} \bigg)^{1/p'} \| \psi \|_{L^p}
\\&\leq C(   \|\psi\|_{W^{1,\I}} + \| \psi\|_{L^p}  )
}
where   $p$ and $p'$ are H\"older conjugates and $p>1$. Since $\psi\in \mathcal D$ this implies $R_iR_j(  \psi)\in L^\I$. Hence
\EQN{
&\bigg|\int_0^T \int u_i (x)u_j(x) \th_R(x_0-x) R_iR_j (   \psi) (x,t)\,dx\,dt \bigg|
\\& \leq 
 CR^3 (   \|\psi\|_{L^\I W^{1,\I}} + \| \psi\|_{L^\I L^p}  ) \|u\|_{L^2_\uloc(0,T)}^2<\I.
}

We now check that the map gives a distribution.  Fix $\la\in\R$ and $\psi_1,\psi_2\in \mathcal D(B\times (0,T))$. Then, by linearity of Riesz transforms and the Lebesgue integral,
\EQN{
\langle \bar p_{\near}^B , \la \psi_1+\psi_2\rangle &= \int_0^T \int u_i (x,t)u_j(x,t) \th_R(x_0-x) R_iR_j  (\la \psi_1+\psi_2)  (x,t)\,dx\,dt
\\&=\la  \int_0^T \int u_i (x,t)u_j(x,t) \th_R(x_0-x) R_iR_j \psi_1 (x,t)\,dx\,dt
\\&+  \int_0^T \int u_i (x,t)u_j(x,t) \th_R(x_0-x) R_iR_j \psi_2 (x,t)\,dx\,dt
\\&= \la \langle \bar p_{\near}^B , \psi_1 \rangle  +\langle \bar p_{\near}^B , \psi_2\rangle.
}
For the same reason,
\EQN{
&\langle \bar p_{\far}^B , \la \psi_1+\psi_2\rangle 
\\&=\int_0^T\int \int (K_{ij}(x-y) - K_{ij}(x_0-y)) (1-\th_R (x_0-y))( u_iu_j)(y,t) \,dy\,  (\la \psi_1+\psi_2)(x,t) \,dx\,dt
\\&=\la \langle \bar p_{\far}^B ,  \psi_1 \rangle +\langle \bar p_{\far}^B ,  \psi_2\rangle.
}
Therefore the map is linear.

Let $\{\psi_k \}_{k=1}^\I \subset C_c^\I (K)$ which converges to zero in the topology of $\mathcal D$ where $K$ is a compact subset of 
$B_R(x_0)\times (0,T)$.
Convergence of $\psi_k$ in the topology on $\mathcal D$ implies $ \|\psi_k\|_{L^\I W^{1,\I}} + \| \psi_k\|_{L^\I L^p}  \to 0 $ as $k\to 0$. Hence
\EQN{
&\bigg|\int_0^T \int u_i (x)u_j(x) \th_{R}(x_0-x) R_iR_j (   \psi_k) (x)\,dx\,dt\bigg|\\& \leq  CR^3 (   \|\psi_k\|_{L^\I W^{1,\I}} + \| \psi_k\|_{L^\I L^p}  ) \|u \|_{L^2_\uloc(0,T)}^2 \to 0.
}
That the remaining term vanishes follows from standard estimates (e.g. \cite{KiSe,KwTs}).
\end{proof}

Lemma \ref{lemma.pressure.distribution} only defines the pressure in $\mathcal D'(B_R(x_0) \times (0,T))$. We need to define a distribution in $\mathcal D'(\R^3\times (0,T))$. To do this, we build upon Lemma \ref{lemma.pressure.distribution} using the following recursive procedure:  
\EQ{\label{def.pressure.dist}
\begin{cases}\langle \bar p (x),\psi\rangle:= \langle \bar p^{B_1(0)} , \psi\rangle  & \text{ if }\psi \in \mathcal D( B_1(0)\times (0,T)),
\\ \langle \bar p (x),\psi\rangle:= \langle \bar p^{B_n(0)} , \psi\rangle + \langle \sum_{k=2}^n c_k , \psi\rangle  & \text{ if }n\geq 2\text{ and }\psi \in \mathcal D( B_{n}(0)\times (0,T)) ,
\end{cases}
}
where
\[
c_{k}{(t)}=-\int K_{ij}(x-y)  (\th_{k} (-y)-\th_{k-1} (-y))( u_iu_j)(y,t) \,dy,
\]
which, after unraveling notation, is the same constant appearing in Theorem \ref{thrm.pressure}.

We need to check that \eqref{def.pressure.dist} defines a distribution.
\begin{lemma}\label{lemma.pressure.distribution2}Assume $u\in L^2_\uloc(0,T)$ for some $T>0$. Let $\bar p$ be defined by \eqref{def.pressure.dist}. Then $\bar p \in \mathcal D'(\R^3\times (0,T))$.
\end{lemma}

\begin{proof}
We abbreviate $\bar p^{B_n(0)}$ by $\bar p^n$ and do the same for the components of $\bar p^n$.  Note that $c_n \in L^1 (0,T)$, which follows from the fact that $u\in L^2_\uloc(0,T)$.

We first  check the map is well defined. Fix $\psi\in \mathcal D'(\R^3\times (0,T))$. Let $n$ be a natural number so that $ \psi\in \mathcal D'(B_n(0)\times (0,T))$. To prove the map is well defined, it suffices to show 
\[
\langle  \bar p^n ,\psi  \rangle+ \langle \sum_{k=2}^n c_k,\psi\rangle = \langle \bar p^{n+1} ,\psi \rangle +\langle \sum_{k=2}^{n+1} c_k,\psi\rangle. 
\]
This reduces to 
\[
\langle \bar p^n -\bar p^{n+1},\psi \rangle =  \int_0^T\int c_{n+1} (t)\psi \,dx\,dt.
\]
We first compute the near-field part of the left hand side which is
\EQ{\label{eq.well-defined.near}
&\int_0^T \int u_i (x,t)u_j(x,t) \th_n(-x) R_iR_j    \psi  (x,t)\,dx\,dt 
\\&- \int_0^T \int u_i (x,t)u_j(x,t) \th_{n+1}(-x) R_iR_j    \psi  (x,t)\,dx\,dt
\\&=\int_0^T \int u_i (x,t)u_j(x,t) (\th_n(-x) -   \th_{n+1}(-x))R_iR_j    \psi  (x,t)  \,dx\,dt.
}
For the far-field part we have
\EQ{\label{eq.well-defined.far}
&\int_0^T\int \int (K_{ij}(x-y) - K_{ij}(-y)) (1-\th_n (-y))( u_iu_j)(y,t) \,dy\,   \psi(x,t) \,dx\,dt
\\&-\int_0^T\int \int (K_{ij}(x-y) - K_{ij}(-y)) (1-\th_{n+1} (-y))( u_iu_j)(y,t) \,dy\,   \psi(x,t) \,dx\,dt
\\&=\int_0^T\int \int (K_{ij}(x-y) - K_{ij}(-y)) (\th_{n+1} (-y)-\th_{n} (-y))( u_iu_j)(y,t) \,dy\,   \psi(x,t) \,dx\,dt
\\&=\int_0^T\int \int K_{ij}(x-y)  (\th_{n+1} (-y)-\th_{n} (-y))( u_iu_j)(y,t) \,dy\,   \psi(x,t) \,dx\,dt
\\&-\int_0^T\int \int  K_{ij}(-y) (\th_{n+1} (-y)-\th_{n} (-y))( u_iu_j)(y,t) \,dy\,   \psi(x,t) \,dx\,dt.
}
Now, fix $x\in \supp\psi\subset B_n$, and consider the integral
\EQN{
\int K_{ij}(x-y)  (\th_{n+1} (-y)-\th_{n} (-y))( u_iu_j)(y,t) \,dy.
}
Because 
$\th_{n+1}-\th_n \equiv 0$ in $B_{2n}(0)$  and $x\in B_n$, we are integrating over the region where $|x-y|\geq n$. This is important because it means that above integral is not principally valued, it is just a convolution operator. Therefore, 
\EQN{
&\int K_{ij}(x-y)  (\th_{n+1} (-y)-\th_{n} (-y))( u_iu_j)(y,t) \,dy  
\\&= p.v. \int K_{ij}(x-y)  (\th_{n+1} (-y)-\th_{n} (-y))( u_iu_j)(y,t) \,dy 
\\&= R_iR_j ( (\th_{n+1}(-\cdot ) - \th_n(-\cdot ) )  u_iu_j ) - \frac 1 3 \delta_{ij}  (\th_{n+1} (-x)- \th_n(-x ))  u_iu_j  (x,t)
\\& =R_iR_j ( (\th_{n+1}(-\cdot ) - \th_n (-\cdot ))  u_iu_j )  ,
} 
where we again used  $\th_{n+1}-\th_n \equiv 0$ in $B_{2n}(0)$ in the last step.
Using this and Fubini's Theorem to move the Riesz transforms (noting that they reduce to non-principally valued convolution operators as observed above), we obtain
\EQN{ 
&\int_0^T\int \int K_{ij}(x-y)  (\th_{n+1} (-y)-\th_{n} (-y))( u_iu_j)(y,t) \,dy\,   \psi(x,t) \,dx\,dt
\\&=\int_0^T\int R_iR_j ( (\th_{n+1}(-\cdot) - \th_n(-\cdot)  )  u_iu_j ) (x,t)  \psi(x,t) \,dx\,dt
\\&=\int_0^T\int (\th_{n+1} - \th_n )(-x)  (u_iu_j )(x,t) (R_iR_j \psi)(x,t) \,dx\,dt.
}
Comparing this to \eqref{eq.well-defined.near} and adding \eqref{eq.well-defined.near} and \eqref{eq.well-defined.far} gives
\EQN{
\langle \bar p^n -\bar p^{n+1},\psi \rangle &= -\int_0^T\int \int  K_{ij}(-y) (\th_{n+1} (-y)-\th_{n} (-y))( u_iu_j)(y,t) \,dy\,   \psi(x,t) \,dx\,dt
\\&=\int_0^T\int c_{n+1}(t)  \psi(x,t)\,dx\,dt,
}
which proves the map is well-defined.

We now check that the map defines a distribution.  It is easy to see that it is finitely valued and linear, so
we only prove continuity. Let $\{\psi_k\}_{k=1}^\I \subset C_c^\I(K)$ where $K$ is a compact subset of $\R^3\times (0,T)$ and $\psi_k\to 0$ in the topology on $\mathcal D(\R^3\times (0,T))$.    Then there exists $n\in \N$ so that $\psi_k\in \mathcal D( B_n(0)\times (0,T ) )$ for all $k$.  We certainly have $\psi_k\to 0$ in the topology on $\mathcal D(B_n(0)\times (0,T))$.  We also have
\[
\langle \bar p ,\psi_k\rangle = \langle \bar p^n ,\psi_k \rangle + \langle \sum_{j=2}^n c_j ,\psi_k\rangle.
\]
Because $\bar p^n$ is a distribution on $B_n(0) \times (0,T)$, 
\[| \langle \bar p^n ,\psi_k \rangle| \to 0.\] On the other hand, since 
\[
 \sum_{j=2}^n c_j \in L^1(0,T),
\]
integrating against $ \sum_{j=2}^n c_j $
defines an element of $\mathcal D'( \R^3\times (0,T) )$. Therefore
\[
|\langle \sum_{j=2}^n c_j,\psi_k\rangle|\to 0,
\]
as well.
This proves that the mapping $\psi\mapsto \langle \bar p, \psi\rangle$ is in $\mathcal D'(\R^3\times (0,T))$.
\end{proof}

\section{The mild formulation}\label{sec.mild.formulation}

In this section we elaborate on  \cite[Remark 3.1]{MaTe} and state some useful facts about the mild formulation. In particular, mild solutions are in $L^1_\uloc (0,T)$ whenever 
$e^{t\Delta}u_0\in L^1_\uloc(0,T)$ and $u\in L^2_\uloc(0,T)$.
This is enough to make sense of the mild formulation as a distribution.

Recall that $e^{t\Delta}\mathbb P  $ is the Oseen tensor. Let $S$ denote its kernel.  We have the following pointwise estimate for $S$,
\EQ{\label{ineq.oseen}
|\pd_t^m\nb_x^k S(x,t)|\le \frac{C_{k,m}}{(|x|+\sqrt t)^{3+k+2m}},}
see the original work of Oseen \cite{Oseen} and the more recent references \cite{LR,MaTe,ShSh}.  Then, given $F=(F_{ij})_{1\leq i,j\leq 3}\in L^2_\uloc(0,T)$ 
we understand the operation
\[
F(x,t)\mapsto \int_0^te^{(t-s)\Delta}\mathbb P \nb \cdot F (x,s)\,ds,
\]  
in $L^1_\uloc(0,T)$ where $e^{t\Delta} \mathbb P$ is the Oseen tensor. 
 
A similar definition is  \cite[Definition 3.1]{MaTe}. The difference between these is that we are considering the operation in space and time as opposed to at a fixed time. This allows us to apply it to elements of $L^2_\uloc(0,T)$ instead of just elements of $L^\I(0,T;L^2_\uloc)$.  The distinction does not lead to technical challenges. Indeed, we will prove the operation is meaningful in the sense of distributions using the above pointwise estimate for the Oseen tensor, which are also used in \cite{MaTe}.
This will allow us to prove that the mild formulation converges in $L^1_\loc(\R^3\times (0,T))$ and, additionally, ends up in $L^1_\uloc(0,T)$ whenever  $u\in L^2_\uloc(0,T)$ and $e^{t\Delta}u_0\in L^1_\uloc$ respectively. 
 This is the content of the next lemma.   A recurring theme is that these estimates match the far-field decay of the kernel in the LPE. We can therefore use methods resembling the pressure estimates  in \cite{KiSe}. Also note that this result is not new; it is  roughly a re-statement of \cite[Lemma 11.3]{LR}.
We include a detailed proof for the sake of completeness.

\begin{lemma}\label{lemma.mild}
If $e^{t\Delta}u_0\in L^1_\uloc(0,T)$
 and $u\in L^2_\uloc(0,T)$  then, letting
\[
\bar u = e^{t\Delta}u_0(x) -\int_0^t e^{(t-s)\Delta}\mathbb P \nb \cdot (u\otimes u)\,ds,
\]
we have $\bar u \in L^1_\uloc(0,T)$.  Consequently $\bar u \in L^1_{\loc}(\R^3\times (0,T))$ and is therefore defined in $\mathcal D'(\R^3\times (0,T))$.
\end{lemma}

\begin{proof}

We will use the following estimate from \cite[{(1.10)}]{MaTe}: For $F\in (L^{p}_\uloc(\R^3))^{3\times 3}$ and $1\leq q\leq p\leq \I$, we have
\EQ{\label{MaTe.estimate}
\|e^{t\Delta} \mathbb P \nb \cdot F\|_{L^p_\uloc}\leq C \big( \frac 1 {t^{1/2}} +  \frac 1 {t^{(3/q-3/p ) /2 +1/2}}  \big) \|F\|_{L^q_\uloc}.
}

We work with a fixed ball centered  at $x_0$ of radius $1$, and rewrite the integral part of $u$ as
\EQN{ 
\int_0^t e^{(t-s)\Delta}\mathbb P \nb \cdot (u\otimes u) \,ds &= \int_0^t e^{(t-s)\Delta}\mathbb P \nb \cdot (u\otimes u \chi_{B_2(x_0)}) \,ds 
\\&+ \int_0^t e^{(t-s)\Delta}\mathbb P \nb \cdot (u\otimes u (1-\chi_{B_2(x_0)}) )\,ds
\\&=: I_{\text{near}} +I_{\text{far}}.
}
For the near-field part, by \eqref{MaTe.estimate},
\EQN{ 
\|  I_{\text{near}}\|_{L^1(B_1(x_0)\times (0,T))}
&\leq \int_0^T \int_0^t \| e^{(t-s)\Delta} \mathbb P\nb \cdot (u\otimes u \chi_{B_2(x_0)}) \|_{L^1_\uloc} \,ds  \,dt
\\&\leq C \int_0^T \int_0^t  \frac {1} {(t-s)^{1/2}}  \| u(s) \chi_{B_2(x_0)}\|_{L^2_\uloc}^2\,ds\,dt
\\&\leq C \int_0^T \| u(s) \chi_{B_2(x_0)}\|_{L^2_\uloc}^2\int_s^T \frac {1} {(t-s)^{1/2}}  \,dt\,ds
\\&\leq C  T^{1/2} \| u\|_{L^2_{\uloc}(0,T)}^2,
}
where we used Tonelli's theorem.
For the far-field part, letting $B=B_1(x_0)$ we have 
\EQN{ 
& \int_0^T \int_{B} \bigg|  \int_0^t e^{(t-s)\Delta}\mathbb P\nb\cdot ( u\otimes u (1-\chi_{B_2(x_0)})  )(s)\,ds\bigg| \,dx\,dt
\\&\leq    \int_0^T\int_B  \int_0^t \int_{|x_0-y|>2}   {\frac C {( |x-y|+\sqrt {t-s})^{4}} }|u\otimes u|(y,s) \,dy\,ds\,dx\,dt
}
where we have used the kernel estimates \eqref{ineq.oseen}.
Note that since $x\in B_1(x_0)$ and $|x_0-y|\geq 2$, we have
$|x_0-y|\leq 2 |x-y|$.
Hence
\EQN{ 
  \|  I_{\text{far}}\|_{L^1(B_1(x_0)\times (0,T))} 
&\leq
   {CT  \int_0^T \int_{|x_0-y|>2}   \frac C {|x_0-y|^4}|u|^2(y,s) \,dy\,ds,}
}
where we dropped the integral in the $x$-variable because nothing depended on $x$ anymore. Thus
\EQN{ 
  \|  I_{\text{far}}\|_{L^1(B_1(x_0)\times (0,T))} 
&\leq C\sum_{i=0}^\I \int_0^T  \int_{2^{i}2<|x_0-y|\leq 2^{(i+1)}2} \frac C {|x_0-y|^4}|u|^2 (y,s) \,dy\,ds
\\&\leq C \sum_{i=0}^\I  \frac 1 {(2^{i+1}2)^4  }\int_0^T  \int_{2^{i}2<|x_0-y|\leq 2^{(i+1)}2}|u|^2 (y,s) \,dy\,ds
\\&\leq C  \sup_{x_0\in \R^3} \int_0^T\| u^2\|_{L^1(B_1(x_0))}(s)\,ds \sum_{i=0} \frac 1 {2^i} 
\leq C  \| u\|_{L^2_{\uloc}(0,T)}^2.
}

Combining the estimates for $I_\text{near}$ and $I_\text{far}$ shows that $\bar u\in L^1_\uloc(0,T)$.
\end{proof}

\section{Bootstrapping scheme}\label{sec.approx}

To use Theorem \ref{thrm.pressure}, we need to approximate solutions to the Navier-Stokes equations by  $L^\I$ vector fields.  This is done by first mollifying the nonlinearity, then obtaining a pressure for the mollified nonlinearity, and finally solving a nonhomogeneous heat equation. In this section we describe this scheme and establish several convergence properties.

Given a vector field $u:\R^3\times (0,\I)\to \R^3$,  $\bar p$   represents a local pressure expansion of $u$. If $u$ is such that $f_{ij}=u_iu_j$
satisfies the assumptions of Theorem \ref{thrm.pressure}, then $\bar p$ is defined as in Theorem \ref{thrm.pressure} at every time $t$.
{It is a function and} satisfies the LPE. By Lemma \ref{lemma.adjoint} it also satisfies the DLPE and agrees the construction  \eqref{def.pressure.dist} in $\mathcal D'$.
If we only know that the given vector field $u$  is in $L^2_\uloc(0,T)$, then we define $\bar p$ as a distribution according to {Lemma \ref{lemma.pressure.distribution} and \eqref{def.pressure.dist}. It satisfies}
Definition \ref{def.localPressure2}, i.e. $\bar p$ satisfies the DLPE.  The same convention is taken for vector fields $u^\e$ and associated pressures $\bar p^\e$.
In this association of $\bar p$ to $u$, $u$ does not need to be a solution of the Navier-Stokes equations.

Let $F=u\otimes u$ be the matrix with entries $F_{ij}=u_iu_j$.
Let $\eta_\e(x,t)= \e^{-4} \eta  (x/\e,t/\e)$, where $\eta$ is a mollifier in $\R^4$ supported in $B(0,1)\times [0,1]$.  
Let
\EQN{
F^\e(x,t) = \eta_\e {\ \mathop{* }_{x,t}\ } (u\otimes u)(x,t),
}
where we are extending $u\otimes u$ by zero to negative times and times greater than $T$, and the convolution is in both space and time.
Then,  $F_{ij}^\e \in C^\I( \R^4  )$. Furthermore, for  $x_0\in \R^3$, $r>0$, and $\e\lesssim r$, we have for every $1\leq p<\I$, 
\EQN{
 \int_{Q_r(x_0)} |   F_{ij}^\e|^p\,dx {\,dt}  &=  \int_{Q_r(x_0)} | \eta_\e * (F_{ij} \chi_{Q_{r}^*(x_0)  })|^p\,dx {\,dt}
\\&\leq  \int_{\R^{3+1}} | \eta_\e * (F_{ij} \chi_{Q^*_{r}(x_0) })|^p\,dx {\,dt},
}
where $Q_r(x_0) = B_r(x_0)\times [0,T]$ and $Q_r^* (x_0) = B_{2r}(x_0)\times [-r,T+r]$.
Therefore, by Young's inequality, if $\e\lesssim r$, then
\EQ{\label{ineq.youngsinequality}
\|  F_{ij}^\e \|_{L^p(Q_r(x_0) )} \leq C  \|  F_{ij} \|_{L^p( Q^*_{r}(x_0))}\qquad (1\leq p<\I).
}
Note that $C$ is independent of $\e$ provided $\e\lesssim r$.   Taking $r=1$, it follows that if 
$u\in L^2_\uloc(0,T)$, then $F^\e \in L^{1}_\uloc(0,T)$.
Under this assumption we also have $F^\e\in L^\I(\R^3\times [0,T])$. Indeed, because $\eta$ has compact support,  if $(x,t)\in \R^3\times (0,T)$, then
\EQN{
  | F_\e(x,t) |&\leq \int_{B_2(0)\times [-1,T+1]} \eta_\e(x-y,t-s)|u (y,s) |^2\,dy\,ds \leq C_\e \|u\|_{L^2_\uloc(0,T)}^2.
}

Since $F_\e(t)\in L^\I(\R^3)$ for each $t\in (0,T)$ and is locally H\"older continuous due to properties of mollifiers, we can apply Theorem \ref{thrm.pressure} to obtain a distributional solution  $p^\e\in BMO$  to the Poisson equation
\EQN{ 
\Delta p^\e = -\partial_i\partial_j F_{ij}^\e.
}
Then, $p^\e=R_iR_j F^\e_{ij}$ in $BMO$.  By Lemma \ref{lemma.adjoint}, $p^\e$ agrees with the construction \eqref{def.pressure.dist} of a pressure satisfying the DLPE referenced in Lemma \ref{lemma.pressure.distribution2} where $u_iu_j$ is replaced by $F_{ij}^\e$.
Furthermore, $p^\e \in L^\I (0,T;BMO)$ since $F^\e\in L^\I(\R^3\times (0,T))$ and the Riesz transforms are bounded from $L^\I$ to $BMO$ (see \cite[p.~156]{Stein}).

 Let $  u_0^\e (x )= (\bar \eta_\e * u_0)(x )$ where $\bar \eta$ is a spatial mollifier  and 
\EQ{\label{uep.def}
u^\e(x,t) =  e^{t\Delta}u_0^\e (x)  - \int_0^t e^{(t-s)\Delta} \mathbb P \nb\cdot   F^\e \,ds.
}
Because $  F^\e, p^\e \in L^\I(0,T;BMO)$, it follows that $u^\e$ solves the heat equation 
\EQ{\label{eq.epsilon.stokes}
\partial_t u^\e -\Delta u^\e +\nb \cdot F^\e   +\nb p^\e=0,
}  
as a distribution \cite[Lemma 3.1]{Kukavica} (see also \cite{GIM}). 
In summary, for every $\e>0$, there exists a pair $u^\e$ and $p^\e$ so that $u^\e$ is mild, $p^\e$ satisfies the DLPE (for $F^\e$) and $u^\e$ solves \eqref{eq.epsilon.stokes}. The proofs of Theorems \ref{thrm.pressureNSE} and \ref{thrm.pressureImpliesMild} will bootstrap these relationships to solutions to the Navier-Stokes equations.

Before addressing this we deal with two preliminary convergence questions.

\begin{lemma}\label{lemma.convergenceToIntegralForm}
Assume $u_0\in {L^1_\loc}(\R^3)$, is divergence free and  $ e^{t\Delta}u_0\in L^1_\uloc(0,T)$.
Let $u\in L^2_\uloc(0,T)$ be given.
{Then for $u^\e$ defined by \eqref{uep.def},}
\[
u^\e(x,t)\to e^{t\Delta}u_0(x)-\int_0^t   e^{(t-s)\Delta}\mathbb P\cdot \nb (u\otimes u)\,ds,
\]
{as $\e \to 0$}
in $L^1_\loc(\R^3\times [0,T])$ and, therefore, in $\mathcal D'(\R^3\times (0,T))$.
\end{lemma}

\begin{proof} 
We first show $e^{t\Delta}u_0^\e\to e^{t\Delta}u_0$ in $L^1_\loc(\R^3\times [0,T])$.  Fix $x_0\in \R^3$ and consider $B_1(x_0)\times (0,T)$.  
We will show
\EQ{\label{u0ep.conv}
\int_0^T\int_{B_1(x_0)}|e^{t\Delta}u_0^\e- e^{t\Delta}u_0|\,dx\,dt \to 0,\quad \text{as}\quad \e \to 0.
}
Let $G_t$ be the kernel of $e^{t\Delta}$, $e^{t\Delta} u_0= G_t*u_0$. We have
\[
 e^{t\Delta} u_0^\e = G_t * (\bar \eta_\e * u_0)=
\bar \eta_\e * (G_t * u_0)=
\bar \eta_\e * (e^{t\Delta}u_0).
\]
Since $e^{t\Delta}u_0\in L^1_\uloc(0,T)$, for almost every $t$ we have $e^{t\Delta}u_0 \in   L^1(B_2(x_0))$. 
By properties of mollifiers, $e^{t\Delta} u_0^\e =\bar\eta_\e * (e^{t\Delta}u_0) \to e^{t\Delta}u_0$ in $L^1(B_1(x_0))$   for a.e. $t$.
Also note that {in $B_1(x_0)$}
\[
\bar \eta_\e * (G_t *u_0) = \bar \eta_\e * ( \chi_{B_2(x_0)} G_t *u_0),
\]
provided $\e\ll 1$.
By Young's inequality we therefore have
\[
\int_{B_1(x_0)}| \bar \eta_\e * (G_t *u_0) |\,dx \leq C\|\bar \eta_\e \|_1 \| G_t * u_0\|_{L^1{(B_2(x_0))}} =C\|\bar \eta_\e \|_1 \| e^{t\Delta}u_0\|_{L^1{(B_2(x_0))}}\in L^1(0,T),
\]
by our assumption that $e^{t\Delta }u_0 \in L^1_\uloc(0,T)$.
Then, by the dominated convergence theorem, we get \eqref{u0ep.conv}.

We next show
\[
\int_0^t e^{(t-s)\Delta}\mathbb P\nb\cdot ( F-F^\e )\,ds\to 0,\quad {\text{as}\quad \e \to 0},
\]
in ${L^1}(B_1(x_0)\times [0,T])$ for any $x_0$ by showing that
\EQ{\label{eq5.9}
	\int_0^T\int_{B_1(x_0)}  \bigg| \int_0^t e^{(t-s)\Delta}\mathbb P\nb\cdot ( F-F^\e )\,ds\bigg| \,dx\,dt
}
can be made arbitrarily small by taking $\e$ sufficiently small.  Let $R>4$ and $R> 2|x_0|$ be given.  We split the term we are bounding as
\EQN{ 
\eqref{eq5.9}
&\leq  \int_0^T\int_{B_1(x_0)} \bigg| \int_0^t e^{(t-s)\Delta}\mathbb P\nb\cdot (( F-F^\e ) \chi_{B_R(0)}  )\,ds\bigg|  \,dx\,dt 
\\&+\int_0^T\int_{B_1(x_0)} \bigg| \int_0^t e^{(t-s)\Delta}\mathbb P\nb\cdot ( (F-F^\e)(1-\chi_{B_R(0)} ) )\,ds\bigg|  \,dx\,dt  {=:J_1+J_2.}
}
Using \eqref{MaTe.estimate} and Tonelli's theorem,   %
\EQN{  
J_1&\leq  C \int_0^T \int_0^t   \| e^{(t-s)\Delta} \mathbb P\nb\cdot (F^\e-F) (s)\chi_{B_R(0)} \|_{L^1_{\uloc}}\,ds\,dt
\\&\leq  C(R) \int_0^T \int_0^t 
\frac 1 {(t-s)^{1/2}} 
\| (F^\e-F) (s)\chi_{B_R(0)} \|_{L^1_\uloc} \,ds\,dt
\\&= C(R) \int_0^T\| (F^\e-F) (s)\chi_{B_R(0)} \|_{L^1_\uloc} \int_s^T \frac 1 {(t-s)^{1/2}}\,dt\,ds
\\&\leq C(R,T) \int_0^T\int_{B_R(0)} |(F^\e-F) (x,s)| \,dx \,ds.
}
By properties of mollifiers, this can be made small by taking $\e$ sufficiently small. 

For the far-field part, using  the Oseen tensor estimate  \eqref{ineq.oseen} 
\EQN{
 {J_2}& \le \int_0^T\int_{B_1(x_0)} \int_0^t  \int_{|y|>R}   \frac C {   (|x-y|+\sqrt {t-s})^{4} }|F-F^\e| (y) \,dy \,ds\,dx\,dt.
}
Since $x\in B_1(x_0)$ and $|y|\geq R \geq |x_0|/2$, 
\EQ{\label{ineq.farfield.R2}
 {J_2}&\leq \int_0^T\int_{B_1(x_0)} \int_0^t \int_{|y|>R}   \frac C {|y|^4}| F-F^\e| (y,s) \,dy \,ds\,dx\,dt
\\
&\leq CT \int_0^T \int_{|y|>R}   \frac C {|y|^4}| F-F^\e| (y,s) \,dy \,ds
\\
&\leq  C T \sum_{i=0}^\I \frac 1 {(2^{i}R)^4} \int_0^T \int_{2^iR <|y|< 2^{i+1}R} | F-F^\e|(y,s) \,dy\,ds
\\&\leq \frac {CT} R \|F-F^\e \|_{L^1_\uloc (0,T)}
{\lec \frac {CT} R (\|F \|_{L^1_\uloc (0,T)}+ \|F^\e \|_{L^1_\uloc (0,T)}).}
}
Using the uniform bound \eqref{ineq.youngsinequality}, we can first take $R$ sufficiently large to make $J_2$ small, then take $\e$ sufficiently small to make $J_1$ small. Hence $\lim_{\e\to 0}\eqref{eq5.9}= 0$. 

We have shown
\[
u^\e \to e^{t\Delta}u_0 - \int_0^t  e^{(t-s)\Delta}\mathbb P \nb\cdot (u\otimes u)\,ds,
\]
in $L^1(B_1(x_0)\times [0,T])$ for every $x_0$. Convergence in $L^1_\loc(\R^3\times [0,T])$ follows.
\end{proof}

On the other hand, we need to prove that $ p^\e$ converges to $\bar p$ in some sense.

\begin{lemma}\label{lemma.pressureConvergence}
Assume   $u\in L^2_\uloc(0,T)$.
Let $p^\e$ be defined as above, and let  $\bar p$ be defined by Lemma \ref{lemma.pressure.distribution} and \eqref{def.pressure.dist}.
Then we have 
\[
  p^\e \to   \bar p,
\]
as distributions in $\cD'(\R^3 \times (0,T))$. Consequently, $\nb   p^\e \to   \nb \bar p $ in $\cD'(\R^3 \times (0,T))$ as well.
\end{lemma}

This lemma makes no assumption on $u_0$, and $u$ need not be a solution to \eqref{eq.NSE}.
\begin{proof}

Let $\phi\in C_c^\I (\R^3\times (0,T))$ be given.
Fix a ball $B= B_n(0)$ and $\tau>0$ so that $B\times [0,T-\tau]$ contains the support of $\phi$.    Then,
\[
\langle   \bar p,  \phi\rangle = \langle \bar p^{n}  ,  \phi		\rangle  + \langle  \td c_n  ,\phi  \rangle,
\]
where $\bar p^n=\bar p^{B_n(0)}$  and $\td c_n$ is the sum appearing in \eqref{def.pressure.dist}.
By Lemma \ref{lemma.adjoint}, we have the corresponding equality for $p^\e$, namely
\[
\langle  p^\e ,   \phi \rangle = 
\langle  p^{\e,n} , \phi \rangle +  \langle  \td c_n^\e  ,\phi  \rangle.
\]

We clearly have
\EQN{ 
&\langle \bar p_{\near}^{n}  -   p_{\near}^{\e,n},   \phi\rangle  
\\&= \int_0^T\int (F_{ij}-F_{ij}^\e)(x,t)(\th_{n}(-x)) R_iR_j  \phi)(x,t)\,dx\,dt \to 0,
} 
by properties of mollifiers because $R_i R_j(\phi)  \in L^\I(\supp \th_{n}   \times (0,T))$ (to see this use Lipschitz continuity to deplete the singularity of the kernel and H\"older's inequality to account for the tail of the kernel) and $\th_{n}$ is compactly supported. 
 
Also by properties of mollifiers, we see that 
\[
\td c_n^\e \to \td c_n,
\]
in $\mathcal D'(0,T)$. Therefore
\[
 \langle   \td c_n - \td c_n^\e  ,\phi  \rangle \to 0.
\]

For the far-field part we let
\[ J^{\e,n}(x,t) = 
\int  (K_{ij}(x-y)-K_{ij} (-y)  ) (1- \th_{n}( -y) ) (F_{ij} -F_{ij}^\e )(y,t)\,dy.
\]
 For $x\in B_n(0)$, we have  
\EQN{
|J^{\e,n}  (x,t) |
&\leq \int_{|y|>n} \frac C {|y|^4} | F-F^\e  | (y,t)\,dy, 
}
which is essentially  the inner integral of \eqref{ineq.farfield.R2} where $n$ has replaced $R$.  Since there is no dependence on $x$ we obtain
\EQN{\label{ineq.farfield.m}
 \langle   \bar p_{\far}^{n}  -   p_{\far}^{\e,n},\phi\rangle  
&=\int_0^T\int J^{\e,n}   \phi \,dx\,dt
\\&\leq C(\phi) \int_0^T \int_{|y|>n}  \frac 1 {|y|^4} | F-F^\e  | (y,t)\,dy		\,dt
\\&\leq \frac {CT} {n} \|F-F^\e\|_{L^1_\uloc(0,T)} \leq 
 \frac {CT} {n}  \big\||F|+|F^\e| \big\|_{L^1_\uloc(0,T)},
}
where we reasoned as in \eqref{ineq.farfield.R2}. 

To conclude, 
we first make the far-field contribution arbitrarily small by taking $n$ sufficiently large and then make the near-field and constant contributions arbitrarily small by taking $\e$ sufficiently small. This shows that 
\[
\langle  p^\e ,\phi \rangle \to \langle \bar p ,\phi \rangle,
\]
for every $\phi\in \mathcal D(\R^3\times (0,T))$, implying $ p^\e \to   \bar p$ in $\mathcal D'(\R^3\times (0,T))$.
\end{proof}

\begin{remark} \label{remark.Lp}
If $u\in L^q_\loc(\R^3\times (0,T))$, $2<q<\I$, and $\bar p$ satisfies the LPE (i.e. in the sense of Definition \ref{def.localPressure}), then the above argument can be adapted to prove that $  p^\e\to   p$ in $\mathcal D'$ using the boundedness of Calderon-Zygmund operators on $L^{q/2}$ to deal with the near-field part. The far-field part is treated the same.
 
\end{remark}

\section{Mild solutions are DLPE solutions}\label{sec.proof.main}

\begin{proof}[Proof of Theorem \ref{thrm.pressureNSE}]

Under the assumptions of Theorem \ref{thrm.pressureNSE}, let $F^\e,u^\e,p^\e,\bar p$ be defined as in Section \ref{sec.approx}. In particular, $\bar p$ is defined by \eqref{def.pressure.dist}. We know by Lemma \ref{lemma.convergenceToIntegralForm} that $u^\e$ converges to $u$ in $L^1_\loc(\R^3\times [0,T] )$.

We first show that $u$ and $\bar p$ satisfy
\EQN{
\partial_t  u -\Delta u +u\cdot \nb u +\nb \bar p = 0;\qquad \nb\cdot   u = 0,
}
with initial data $u_0$ in the sense of distributions, i.e., for any $\phi \in (C_c^\I(\R^3\times (0,T)))$,  
\EQ{\label{eq.weak.form}
\int_0^T\int u\cdot (\partial_t   {+}\Delta )\phi +  u_iu_j \partial_i \phi_j)\,dx\,dt  {-} \langle \nb \bar  p,\, \phi \rangle =0,
}
where we have summed over $i,j$.
By \cite[Lemma 3.1]{Kukavica}, ${u^\e}$ solves   
\EQ{ \label{eq.heat}
\partial_t u^\e - \Delta u^\e +\nb \cdot F^\e +\nb p^\e=0,
}
distributionally.   
Using \eqref{eq.heat}, the left hand side of \eqref{eq.weak.form} is equal to 
\EQN{
\int_0^T\int (u-u^\e)\cdot (\partial_t   {+}\Delta )\phi + ( F_{ij} -F_{ij}^\e) \partial_i \phi_j \,dx\,dt+\langle \bar  p -p^\e,\, \nb\cdot \phi \rangle.
}
It suffices to show this  vanishes as $\e\to 0$.

The first term is $\int_0^T\int (u-u^\e)\cdot (\partial_t  {+}\Delta )\phi \,dx\,dt$, which vanishes by Lemma \ref{lemma.convergenceToIntegralForm} and the fact that  $u$ is mild.
The nonlinear term,
\[
\int_0^t \int ( F_{ij}-  F_{ij}^\e ) \partial_i \phi_j \,dx\,dt,
\]
vanishes by properties of mollifiers.  The pressure term vanishes by Lemma \ref{lemma.pressureConvergence}.
This proves the main part of the theorem.

Now assume  that $u\in L^q_\loc(\R^3\times [0,T))$ for some $q>2$. 
Let $\pi$ be defined as in Theorem \ref{thrm.pressure}. We estimate $\pi$ in $L^{q/2}_\loc(\R^3\times (0,T))$.  Let $\de>0$ be given.  For simplicity we do this over concentric shells.  Fix $n\in \N$.  For $x\in B_{n+1}\setminus B_n$ we  write $\pi$ as  
\EQN{ 
\pi (x,t)
=p_{1}+ p_{2} +\td c_{n+1}(t),
}
where $p_1$ is the Calderon-Zygmund part of $G_{ij}^B (u_iu_j)$ from \eqref{eq.pressureexpansion} where $B=B_{n+1}(0)$, $p_2$ is the remaining term in \eqref{eq.pressureexpansion}  and $\td c_{n+1}$ comes from the comments following Theorem \ref{thrm.pressure}.
By the Calderon-Zygmund theory, we have    
\[
\|  p_{1} \|_{L^{q/2}(0,T-\de;L^{q/2} (B_{n+1}\setminus B_n ) )}\leq \|  u \|_{L^{q}(0,T-\de;L^{q} (B_{n+1}(0)) )}^2.
\]
The right hand side is finite since $u\in L^q_\loc(\R^3\times (0,T))$.
On the other hand, as before
\[
| p_{2} (x,t)|\leq \int_{|y|>2n} \frac 1{|y|^4} |u(y,t)|^2\,dy, 
\]
implying 
\[
\int_0^{T-\de}\int_{B_{n+1}\setminus B_n} |p_2 (x,t)|^{q/2} \,dx\,dt \leq C T n^3 \|u\|_{L^2_\uloc(0,T)}^q.
\]
Finally we consider the constant $\td c_{n+1}(t)$.  From Section \ref{sec.pressureformula}, we know that
\[
\td c_{n+1}(t)=\int \big( K_{ij}^{2} (-y)   - K_{ij}^{2(n+2)}(-y)\big) u_iu_j (y,t)\,dy \leq C(n) \|u (t) \|_{L^q(B_{4n}(0))}^{2/q}.
\]
This is certainly in $L^{q/2}(0,T)$ because $u\in L^{q/2}_\loc(\R^3\times [0,T)  )$.   Since the above estimates hold for all $n$ and $\de>0$, 
it follows that  $\pi\in L_\loc^{q/2}(\R^3\times [0,T))$.
 
Note that $p^\e$ satisfies both the LPE and DLPE by Lemma \ref{lemma.adjoint}.  Since $\pi\in L_\loc^{q/2}(\R^3\times [0,T))$, $\pi$ is defined in $\mathcal D'$.  We have $ p^\e \to  \bar p$ in $\mathcal D'$ by Lemma \ref{lemma.pressureConvergence} and $ p^\e \to  \pi$ in  $\mathcal D'(\R^3\times (0,T))$ by
Remark \ref{remark.Lp}. It follows that $\bar p =  \pi$ in $\mathcal D'(\R^3\times (0,T))$ and therefore $\bar p = \pi$ a.e.~in $\R^3\times (0,T)$. Hence,  $\bar p \in L_\loc^{q/2}(\R^3\times [0,T))$.

Finally we show that if $u\in \mathcal N(u_0)$,  then $u$ and $\bar p$ satisfy the local energy inequality. 
Consequently, $u$ is a local Leray 
 solution with {\emph{this choice of pressure}} $\bar p$.
If $u\in  \mathcal N(u_0)$, then there exists $p$ so that $u$ and $p$ satisfy Definition \ref{def:localLeray}.  Since $u$ satisfies \eqref{eq.NSE} for both $p$ and $\bar p$, it follows that $\nb p=\nb \bar p$ as distributions.
Because $u$ and $p$ satisfy the local energy inequality, it suffices to show
\[
\int_0^T\int p( u\cdot \nb \phi   ) \,dx\,dt = \int_0^T\int \bar p(u\cdot \nb \phi )\,dx\,dt,
\]
where $\phi\in C_c^\I(Q)$ for some parabolic cylinder compactly embedded in $\R^3\times (0,\I)$, as this will establish the local energy inequality for $u$ and $\bar p$.
 
Note that $u\cdot \nb \phi \in L^{3}(Q)$.  We can therefore approximate $u\cdot \nb \phi |_Q $ in $L^{3}( \R^3\times (0,\I))$ by a sequence $ \nb\cdot \Phi_n$ so that $\Phi_n \in C_c^\I (  Q^*)$ where $Q^*$ is a neighborhood of $Q$ that is still compact and bounded away from $t=0$. 

Since $\nb p=\nb \bar p$ as distributions and $ \Phi_n \in C_c^\I(\R^3\times (0,\I))$, 
\EQN{
\int_0^T\int (p-\bar p )\nb\cdot \Phi_n \,dx\,dt=0,
}
for all $n$.  So,
\EQN{
\int_0^T\int (p-\bar p ) ( u  \cdot\nb \phi )\,dx\,dt=
\int_0^T\int (p-\bar p ) (u  \cdot\nb \phi  - \nb\cdot \Phi_n )\,dx\,dt.
}
This vanishes because $p,\,\bar p\in L^{3/2}_\loc (\R^3\times [0,T])$  and $\nb\cdot \Phi^n\to u  \cdot\nb \phi$ in $L^3(\R^3\times (0,T))$.    

It follows that $u$ and $\bar p$ satisfy all elements of Definition \ref{def:localLeray}.  The structure of $\bar p$ implies $u$ is, additionally, a local energy solution.
\end{proof}

\section{DLPE solutions are mild solutions}\label{sec.proof.main.2}

\begin{proof}[Proof of Theorem \ref{thrm.pressureImpliesMild}]
We work under the assumptions of Theorem \ref{thrm.pressureImpliesMild} and let $F^\e,u^\e,p^\e$ be as in Section \ref{sec.approx}. So, the pressure $p$ associated with $u$ is defined and satisfies the DLPE.   Note that this may not agree with the construction \eqref{def.pressure.dist} in $\mathcal D'$. However, if $\bar p$ is defined according to \eqref{def.pressure.dist}, then $\nb p= \nb \bar p$ in $\mathcal D'$.  This is because the definitions of $p$ and $\bar p$ only differ by constants.

Note that $\nb p^\e\to \nb p$ as distributions by Lemma \ref{lemma.pressureConvergence}.
By \cite[Lemma 3.1]{Kukavica}, {as for \eqref{eq.epsilon.stokes}}, 
\[
u^\e(x,t)= e^{t\Delta}u_0^\e (x) -\int_0^te^{(t-s)\Delta} \mathbb P\nb\cdot F^\e\,ds,
\]
 is a distributional solution to the non-homogeneous heat equation  
\[
\partial_t u^\e - \Delta u^\e = - \nb \cdot F^\e -\nb p^\e.
\]
Let
\[
\bar u(x,t)=e^{t\Delta}u_0  (x) -\int_0^te^{(t-s)\Delta} \mathbb P\nb\cdot (u\otimes u)\,ds.
\]
By Lemma \ref{lemma.convergenceToIntegralForm}, $u^\e\to \bar u$ as distributions.  As in the proof of Theorem \ref{thrm.pressureNSE}, we see that $\bar u$ is a distributional solution
to the non-homogeneous heat equation
\[
\partial_t \bar u - \Delta \bar u = - u\cdot\nb u -\nb p.
\]
Also note $\bar u\in L^1_\uloc(0,T)$ by Lemma \ref{lemma.mild}.
Thus
\[
\partial_t \bar u - \Delta \bar u +u\cdot \nb u +\nb p = \partial_t u -\Delta u +u\cdot
\nb u +\nb p,
\]
implying 
\[
\partial_t (\bar u-u)-\Delta (\bar u -u)=0,
\]
all in the distributional sense.  We will show $u=\bar u$ using uniqueness for the heat equation but first need to show that $\bar u$ converges to $u_0$ in $L^1_\loc (\R^3)$.
 
Note that we have assumed for every compact set $K$ that 
\[
\int_K |u(x,t)-u_0(x)|^2\,dx \to 0.
\]
This implies that\footnote{It is in fact sufficient to have $u\to u_0$ in an $L^1_\loc$ sense if we also have 
$\limsup_{t\to 0^+} \| u(t)\|_{L^2(K)}<\I$ for every compact $K$.} 
\[
\int_K |\bar u(x,t)-u_0(x)| \,dx \to 0,
\]
as we now show.  Certainly 
\[
\int_K |e^{t\Delta}u_0(x)-u_0(x)| \,dx \to 0,
\]
and so it suffices to show 
\[
\int_K |I(x,t)| \,dx \to 0,
\]
as $t\to 0^+$
where 
\[
I(x,t)=\int_0^te^{(t-s)\Delta}\mathbb P\nb \cdot (u\otimes u)\,ds.
\]
Let $\varepsilon>0$ be given and let $I = I_{\text{near}}^R + I_{\text{far}}^R$, where 
\[
 I_{\text{near}}^R= \int_0^t e^{(t-s)\Delta}\mathbb P \nb \cdot (u\otimes u \chi_{B_{2R}(0)}) \,{ds,}
\]
and 
\[
 I_{\text{far}}^R= \int_0^t e^{(t-s)\Delta}\mathbb P \nb \cdot (u\otimes u(1- \chi_{B_{2R}(0)})) \,{ds.}
\]
By taking $R\in\N$ sufficiently large we have $K\subset B_R(0)$.

Following a familiar argument we have for any $\ga>0$ that
\EQN{
&\sup_{0<t<\gamma} \int_K |I_{\text{far}}^R(x,t)| \,dx 
\\&\leq C \sup_{0<t<\gamma}\int_K  \int_0^t \int_{|y|>R}   \frac C {|y|^4}|u|^2 (y) \,dy\,ds\,dx
\\&\leq C(K)\sup_{0<t<\gamma} \sum_{i=0}^\I \int_0^t  \int_{2^{i}R<|y|\leq 2^{(i+1)}R} \frac C {|y|^4}|u|^2  (y,s) \,dy\,ds
\\&\leq C \sup_{0<t<\gamma} \sum_{i=0}^\I  \frac 1 {(2^{i+1}R)^4  }  \int_0^t\int_{2^{i}R <|y|\leq 2^{(i+1)}R}|u|^2  (y) \,dy\,ds
\\&\leq C   R^{-1}  \sup_{x_0\in \R^3} {\int_0^T} \int_{B_1(x_0)} |u|^2(x,t)\,dx\,dt.
}
By taking $R$ large this can be made less than $\varepsilon/2$. Fix $R$ so that this is true.
For the other part, {with $\ga<1$}
\EQN{ 
\sup_{0<t<\gamma} \|  I_{\text{near}}^R\|_{L^1(K)}
&\leq \sup_{0<t<\gamma} \int_0^t \| e^{(t-s)\Delta} \mathbb P\nb \cdot (u\otimes u \chi_{B_{2R}(0)}) \|_{L^1} \,ds  
\\&\leq C\sup_{0<t<\gamma} \int_0^t\frac {1} {(t-s)^{1/2}} \| u(s) \chi_{B_{2R}(0)}\|_{L^2}^2\,ds.
}
Since $\| u(s)-u_0\|_{L^2 (B_{2R}(0))} \to 0$, there exists $\gamma_*<1$ so that $\sup_{0<s<\gamma_*}\| u(s) \|^2_{L^2 (B_{2R}(0))}< 2\|u_0\|^2_{L^2(B_{2R}(0))}$.  Hence, for $\gamma<\gamma_*$,
\EQN{ 
\sup_{0<t<\gamma} \|  I_{\text{near}}^R\|_{L^1(K)}&\leq 2C\gamma^{1/2}\|u_0\|^2_{L^2(B_{2R}(0))}.
}
By taking $\gamma$ sufficiently small we can ensure 
\[
\sup_{0<t<\gamma} \|  I_{\text{near}}^R\|_{L^1(K)} <\varepsilon/2.
\]
This proves that 
$\int_K |\bar u(x,t)-u_0(x)| \,dx \to 0$.\footnote{This conclusion also appears in \cite[Lemma 11.3]{LR}. We include the details for completeness.}

We will now prove that $u=\bar u$.  Let $w=u-\bar u$ and let $\eta_\e$ be the space-time mollifier from Section \ref{sec.approx}, {supported in $B_\e(0)\times[0,\e]$.}  Let $w_\e = \eta_\e * w$ where $w$ is extended by zero when required.  We will check that $w_\e$ is a bounded solution to the heat equation in the distributional sense on $\R^3\times (0,T)$ and $w_\e |_{t=0}=0$. By uniqueness, it follows that $w_\e = 0$ and, in turn, $w=0$.

We first show that $w_\e$ solves the heat equation. Let $\psi\in C_c^\I(\R^3\times (0,T))$. Note that $\eta_\e$ is zero except possibly at times in {$[0,\e]$}. Then
\EQN{
&\int_0^T \int (\partial_t +\Delta)\psi(x,t ) w_\e(x,t)\,dx\,dt 
\\&=\int_0^\e \int \eta_\e(y,s) \int_0^T\int	(\partial_t +\Delta)\psi(x,t) w(x-y,t-s)							\,dx\,dt	\,dy\,ds
\\&= \int_0^\e \int \eta_\e(y,s) \int_0^T\int	(\partial_\tau +\Delta_z)\psi(z+y,\tau+s) w(z,\tau)							\,dz\,d\tau	\,dy\,ds,
}
where we have let $z=x-y$ and $\tau=t-s$ and used the fact that $w(z,\tau)=0$ for $\tau\leq 0$. We would like to conclude that the inside integral is zero but cannot immediately because the test function (as a function of $z,\tau$) is not compactly supported in $\R^3\times (0,T)$. Indeed, it is possibly nonzero at $\tau=0$.
For $\ga>0$ let $\si_\ga(t) =1 $ for $t\in [0,\ga/3]$, decrease to $0$ with gradient bounded in absolute value by $C\ga^{-1}$ over $[\ga/3, 2\ga/3]$ and equal zero for $t\geq 2\ga/3$. We take $C$ to be independent of $\ga$, e.g. define $\si_\ga = \si_1( x /\ga  )$ for a  fixed function $\si_1$.  Let $\psi=\psi_1 +\psi_2$ where $\psi_1(z,\tau) = \si_\ga(\tau) \psi(z+y,\tau+s)$ and $\psi_2(z,\tau) = \psi(z+y,\tau+s)-\psi_1(z,\tau)$.  Then, 
\EQN{
\int_0^\e \int \eta_\e(y,s) \int_0^T\int	(\partial_\tau +\Delta_z)\psi_2(z,\tau) w(z,\tau)							\,dz\,d\tau	\,dy\,ds = 0,
}
because $\psi_2$ is a test function on $(0,T)\times \R^3$. On the other hand, let $B$ contain the spatial support of $\psi$. Then, ${\psi_1}(z,\tau) w(z,\tau)\neq 0$ implies $z+y=x\in B$, i.e., $z\in (B-y)$. Since $|y|<\e$ and assuming $\e\ll 1$, we therefore have ${\psi_1}(z,\tau) w(z,\tau)\neq 0$ only if $z\in 2B$, regardless of $y$.  Hence,
\EQN{
&\bigg|\int_0^\e \int \eta_\e(y,s) \int_0^T\int	(\partial_\tau +\Delta_z)\psi_1(z,\tau) w(z,\tau)							\,dz\,d\tau	\,dy\,ds\bigg| 
\\&\leq \frac C \ga  \int_0^\ga\int_{2B}	 |w(z,\tau)	|		\,dz\,d\tau	 + C \int_0^\ga \int_{2B} |w(z,\tau)|\,dz\,d\tau
\\&\leq C(1+\ga) \sup_{0<\tau <\ga} \int_{2B} |w(z,\tau)|\,dz
\\&\leq C(1+\ga) \sup_{0<\tau <\ga} \bigg(  \int_{2B} |u(z,\tau)-u_0(z)|\,dz + 		\int_{2B} |\bar u(z,\tau)-u_0(z)|\,dz 	\bigg).
}
By assumption, the above clearly vanishes as $\ga\to 0^+$. 
This implies  
\[
\int_0^T \int (\partial_t +\Delta)\psi(x,t ) w_\e(x,t)\,dx\,dt =0,
\]
for all $\psi\in C_c^\I(\R^3\times (0,\I))$ and so $w_\e$ solves the heat equation distributionally.
Furthermore, we have
\[
w_\e(x,0)= \int_\R \int_{\R^3}  \eta_\e(x-y,0-s) w(y,s)\,dy\,ds.
\]
The integrand is nonzero only if $0<0-s<\e$, in which case $s<0$ and so $w(y,s)=0$. This implies $w_\e(x,0)=0$.  That $w_\e$ is bounded follows from the fact that  $w\in L^1_\uloc(0,T)$.
We have thus shown that $w_\e$ is a bounded distributional solution of the heat equation with initial data identically zero. It follows that $w=0$ {and $u=\bar u$}.
\end{proof} 
 
\begin{remark} \label{remark.unitary.scale2}
We can clearly weaken the continuity assumption to the following: for every compact set $K$, 
\[
\lim_{t\to 0^+} %
\| u (t) - u_0\|_{L^1(K)} = 0,
\]
however we also need to assume $\limsup_{t\to 0} \|u(t)\|_{L^2(K)}<\I$.

On the other hand, the continuity assumption at $t=0$ can be replaced by $u\in L^\I(0,T;L^2_\uloc)$. This would imply $\bar u \in L^\I(0,T;L^1_\uloc)$ and, therefore, mollifying only in space, we get $w_\e$ is bounded and $w_\e(x,0)=0$ immediately.  

\end{remark}

\section{Applications}\label{sec.applications}

In this section we give applications of Theorems \ref{thrm.pressureNSE} and  \ref{thrm.pressureImpliesMild} to highlight their usefulness.   

\subsection{An improved uniqueness criteria}  

This application highlights the fact that, due to Theorem \ref{thrm.pressureImpliesMild}, we can use properties of mild solutions to study local energy solutions.
 
In \cite{BT8}, the authors establish a local uniqueness criteria \cite[Theorem 1.7]{BT8}  which involves several assumptions on the initial data, namely   $u_0\in E^2 = \overline{C_c^\I}^{L^2_\uloc}$ or 
\[
\lim_{R\to \I} \sup_{x_0\in \R^3} \frac 1 {R^2} \int_{B_R(x_0)}|u_0(x)|^2\,dx = 0.
\]
The assumptions on the initial data imply the solution in view is mild (as is checked in \cite{BT8}), which is used to prove uniqueness.  However, Theorem \ref{thrm.pressureImpliesMild} states that every local energy solution is mild, and therefore these assumptions are overkill.  This leads to the following improvement of \cite[Theorem 1.7]{BT8}.

\begin{theorem}\label{thrm.uniqueness1}
Assume $u_0\in L^{2}_\uloc$ and is divergence free.   Let $u$ and $v$ be local energy solutions with initial data $u_0$.  There exist universal constants $\e>0$ and $c_1>0$ such that,
 if for some $R>0$,
\[
\sup_{0<r\leq R} \sup_{x_0\in \R^3}\frac 1 r \int_{B_r(x_0)} |u_0|^2\,dx \leq \e,
\]
then $u=v$ as distributions  on $\R^3\times (0,T)$, $T=c_1R^2$.
\end{theorem}
 
To prove Theorem \ref{thrm.uniqueness1}, it suffices to inspect \cite[Proof of Theorem 1.7]{BT8}.

\subsection{A regularity criteria in dynamically restricted local Morrey spaces}

This application highlights the fact that, due to Theorem \ref{thrm.pressureNSE}, we can use properties of local energy solutions to study mild solutions.  To illustrate this, we give a new proof of a recent result of  Gruji\'c and Xu \cite{GX}.  In \cite{GX},  Gruji\'c and Xu establish a regularity criteria for smooth mild solutions which are possibly blowing up at time $T_0$ and belong to $L^\I(\R^3 \times  [0,T)  )$ for all $T<T_0$.  The regularity criteria is novel in that it only involves Morrey type quantities computed at scales above a time-dependent threshold. This is similar to a result in \cite{BG} that used Littlewood-Paley modes  to determine similar scales in a Besov space context.  A remarkable aspect of the argument in \cite{GX} is that it only depends on Gruji\'c's geometric regularity criteria from  \cite{Gr12}, which is self-contained in that it does not rely on previous regularity results for the Navier-Stokes equations as does ours.

The following theorem is essentially an analogue of \cite[Theorem 4.1]{GX}.

\begin{theorem}\label{thrm.reg}
Assume $e^{t\Delta}u_0,u\in L^2_\uloc(0,T_0)\cap  C_{wk}((0,T_0);L^\I)$, $T_0>0$, $u$ and $u_0$ are divergence free, and $u$ is a mild solution in the sense of Lemma \ref{lemma.mild} on $\R^3\times (0,T_0)$ 
that satisfies the restriction property: for any $t_0\in (0,T_0)$, $u(t)|_{t\in (t_0,T_0)}$ is a mild solution with initial data $u(t_0)$.\footnote{
These assumptions are satisfied if $u\in L^\I([0,T_0);L^\I)$ is a strong, mild solution as in \cite{GX}.  
}
For $0<t<T_0$, let $r(t)=\sqrt{T_0-t} /(2\sqrt{c_0})$ where $c_0$ is the constant appearing in \eqref{ineq.apriorilocal}.  There exists a small universal constant $\e>0$ so that if there exists
$t \in (0,T_0)$ so that
\[
\sup_{x_0\in \R^3} \frac 1 {r(t)} \int_{B_{r(t)}(x_0)} |u(x,t)|^2\,dx <\e,
\]
then $u$ can be extended to a bounded solution past time $T_0$.
\end{theorem}

We now connect this to \cite[Theorem 4.1]{GX} by explaining how $r(t)$ relates to the $L^\I$ norms of $u$ and the vorticity $\om$. For the velocity, if $T_0$ is a singular time, then $\|u(t)\|_\I \geq c_1/ (T_0-t)^{1/2}$ for a small constant $c_1$, implying  $r(t) \geq c\|u(t)\|_{\I}^{-1}$. This leads to a velocity-level analogue of  \cite[Theorem 4.1]{GX}.  

The scales in  \cite[Theorem 4.1]{GX} are, however, pegged to the $L^\I$ norm of $\om$.  Theorem \ref{thrm.reg} can easily be extended to give this by additionally assuming $\om_0\in L^2\cap L^\I$. If $u$ blows up at time $T_0$, then so does $\om$. On the other hand, since $u$ is smooth prior to $T_0$, $T_0$ is the first possible blow up time for $\om$.  If $\om$ blows up at time $T_0$ and the other assumptions in   \cite[Theorem 4.1]{GX} are met, then $r(t)\geq c_2 \|\om(t)\|_{L^\I}^{-1/2}$ for a small constant $c_2$ or else we could re-solve the vorticity system using a local well posedness result  from \cite{BFG} to obtain boundedness of the vorticity past time $T_0$. The local well-posedness result from \cite{BFG} uses $\om_0\in L^2\cap L^\I$.  The range of scales over which smallness is needed at certain times $t$ in \cite{GX} is $[ c_*\|\om(t)\|_{L^\I}^{-1/2}  , 1]$ for a constant $c_*$ appearing in the argument from \cite{GX}. Up to a possible revision of the constant $c_*$, the scale $r(t)$ appearing in Theorem \ref{thrm.reg} is in the range from \cite{GX}.

Theorem \ref{thrm.reg} improves \cite[Theorem 4.1]{GX} slightly in that  we do not need to restrict the times considered to be related to escape times (times $t$ so that $\|u(s)\|_\I \geq \|u(t)\|_\I$ for  $t\leq s<T_0$) and we do not need smallness to hold across a range of scales, it is sufficient at a single scale.  On the other hand, \cite[Theorems 4.2 and 4.3]{GX} illustrate that the approach in \cite{GX} can be extended to other Morrey-type spaces (ours is confined to $L^2$-based norms) as well as the more general $Z_\al$ framework of \cite{BFG} (see also \cite{GX2} for an interesting extension of the ideas in \cite{BFG}).

\begin{proof} 
Let $u$ be mild with data $u_0$ and satisfy $u, e^{t\Delta}u_0\in  L^2_\uloc(0,T_0)\cap   C_{wk}((0,T_0);L^\I)$.   We also assume that $u$ satisfies the {restriction property}.

For $t_0\in (0,T_0)$, let $\td u = u|_{t\in (t_0,T_0)}$. Then, $\td u$ is a local energy solution with initial data $u(t_0)$ as we now check.  
{Note that $\td u\in C_{wk}([t_0,T_0);L^\I)$, %
which is a uniqueness class for mild solutions, see  \cite{GIM} and also \cite{Kukavica,MaTe}.} Therefore $\td u$ agrees with the strong solutions constructed in \cite{GIM,Kukavica,KOT,Xu} for initial data $u(t_0)$ {on $\R^3 \times (t_0,T_1)$, for some $T_1 \in (t_0, T_0]$ with $T_1 -t_0 \ge C(\norm{u} _{L^\infty(t_0, \frac 12(t_0+T_0);\, L^\infty)})>0$
 if $T_1 \le \frac 12(t_0+T_0)$.}  This implies $(t-t_0)^{1/2} \nb \td u \in L^\I(t_0,{T_1(t_0)};L^\I)$ 
 (see \cite{KOT}) and, since this is true for all $t_0$, it follows that $\nb u\in L^\I_\loc((0,T_0);L^\I)$. Therefore, $u\in L^\I(t_0,T_0-{\de}; L^2_\uloc)$ and $\nb u\in  L^2_\uloc(t_0,T_0-{\de})$ {for any $0<\de\ll 1$}. The LPE on $(t_0,T_0)$ is satisfied because $\td u$ is mild by the restriction property, gradient estimates  and Theorem \ref{thrm.pressureNSE}. This also ensures that $u$ and the pressure satisfy \eqref{eq.NSE} in $\mathcal D'$.  By weak time continuity we have for any $w\in L^2$ with compact support that the function 
\[
t\mapsto \int u(x,t)\cdot w(x) \,dx ,
\]
is continuous. The local energy inequality follows because $\td u\in L^4_\loc( \R^3\times (0,T_0))$.   {Smoothness of strong solutions implies that $\int_{K}|u(x,t)- u(x,t_0)|^2\,dx\to 0$   as $t\to t^+$ whenever $K$ is compact. }  This all implies $\td u$ is a local energy solution with data $u(t_0)$ on $\R^3\times [t_0,T_0 {-\delta}]$ for all $0<\delta\ll 1$.

Using the above, for $t\in (0,T)$, we consider $u$ on $(t,T {-\delta})$ as a local energy solution evolving from $u(t)$ for every $0<\de\ll 1$. Hence, if $\e_0$ is the universal constant from the Caffarelli-Kohn-Nirenberg type theorem in \cite{L98}, $\e<\e_0$ and for some $\rho>0$
\[
\sup_{x_0\in \R^3}\frac 1 {\rho} \int_{B_\rho(x_0)} |u(x,t)|^2\,dx <\e,
\]
then
we can use \cite[Lemma 2.1]{BT8} (see also \cite{JiaSverak-minimal} and \cite[Lemma 3.5]{KMT}  when $u_0\in E^2$) to obtain 
\[
\sup_{x_0\in \R^3}{\frac 1{\rho^2}} \int_t^{t+{c_0}\rho^2} \int_{B_\rho(x_0)} |u|^3\,dx\,ds+\sup_{x_0\in \R^3}\frac 1{\rho^2} \int_t^{t+{c_0}\rho^2} \int_{B_\rho(x_0)} |p-c_{x_0,\rho}(s)|^{3/2}\,dx\,ds< C\e^{3/2},
\]
provided $(t,t+{c_0}\rho^2)\subset (0,T_0)$.
This implies, by the Caffarelli-Kohn-Nirenberg criteria type theorem in \cite{L98}, that 
\[
\| u\|_{L^\I (\R^3\times (t+3c_0\rho^2/4,t+c_0\rho^2))} <{\frac 1{\rho}} \si(\e),
\] 
where $\si(\e)$ is a function of $\e$ satisfying $\si(\e)\to 0$ as $\e\to 0$.
Let $t_\rho = T_0-4  c_0\rho^2$, so that $r(t_\rho) = \rho$.  Applying the above observations implies, if there exists $0<\rho< \sqrt{T_0/( {4c_0})}$ so that
\[
\sup_{x_0\in \R^3} \frac 1 \rho \int_{B_\rho(x_0)} |u(x,t_\rho)|^2\,dx <\e,
\]
then 
\[
\|u \|_{L^\I ( \R^3 \times [T_0-13 c_0\rho^2/4,T_0-3 c_0\rho^2])}< {\frac 1{\rho}} \si(\e).
\]
We can re-solve starting at $T_0-3{c_0} \rho^2$ to obtain a new, strong solution on a time interval of length equal to $C{\rho^2}/\si(\e)^2$ by the local existence theory for small mild solutions in $L^\infty$; see \cite[Theorem 1.1 and its Remark (iii)]{GIM} as well as \cite[Proposition 3.2]{Kukavica}. By the uniqueness of mild solutions in $C_{wk}(I;L^\infty)$ where $I$ is a time interval,  the strong solution agrees with $u$ at times when both are defined and the time interval of existence of the strong solution extends up to at least time $T_0-3 {c_0}\rho^2 +C{\rho^2}/\si(\e)^2$. This is bigger than $T_0$ provided {$\si(\e)<\sqrt {C/(3{c_0})}$}. Letting $t$   in the theorem equal $t_\rho$ we have $\rho =r(t)$ and the proof is complete.
\end{proof}

\section*{Acknowledgments}
The research of Bradshaw was partially supported by the Simons Foundation.
The research of Tsai was partially supported by NSERC grant RGPIN-2018-04137.

 Zachary Bradshaw, Department of Mathematics, University of Arkansas, Fayetteville, AR 72701, USA;
 e-mail: zb002@uark.edu
 \medskip
 
 Tai-Peng Tsai, Department of Mathematics, University of British
 Columbia, Vancouver, BC V6T 1Z2, Canada;
 e-mail: ttsai@math.ubc.ca

\end{document}